\documentclass{amsart}

\usepackage{amsmath}
\usepackage{amsfonts}
\usepackage{amssymb,enumerate}
\usepackage{amsthm}
\usepackage[all]{xy}
\usepackage{hyperref}

\newtheorem{lem}{Lemma}[section]
\newtheorem{cor}[lem]{Corollary}
\newtheorem{prop}[lem]{Proposition}
\newtheorem{thm}[lem]{Theorem}

\newtheorem{Defn}[lem]{Definition}
\newtheorem{Ex}[lem]{Example}
\newtheorem{Question}[lem]{Question}
\newtheorem{Property}[lem]{Property}
\newtheorem{Properties}[lem]{Properties}
\newtheorem{Discussion}[lem]{Remark}
\newtheorem{Construction}[lem]{Construction}
\newtheorem{Notation}[lem]{Notation}
\newtheorem{Fact}[lem]{Fact}
\newtheorem{Notationdefinition}[lem]{Definition/Notation}
\newtheorem{Remarkdefinition}[lem]{Remark/Definition}
\newtheorem{Subprops}{}[lem]
\newtheorem{Para}[lem]{}

\newenvironment{defn}{\begin{Defn}\rm}{\end{Defn}}
\newenvironment{ex}{\begin{Ex}\rm}{\end{Ex}}
\newenvironment{question}{\begin{Question}\rm}{\end{Question}}

\newenvironment{disc}{\begin{Discussion}\rm}{\end{Discussion}}

\newcommand{\cat}[1]{\mathcal{#1}}
\newcommand{\catd}{\mathsf{D}}

\newcommand{\cata}{\cat{A}}

\newcommand{\catb}{\cat{B}}

\newcommand{\pd}{\operatorname{pd}}

\newcommand{\depth}{\operatorname{depth}}	
\newcommand{\rank}{\operatorname{rank}}	

\newcommand{\edim}{\operatorname{edim}}

\newcommand{\ann}{\operatorname{Ann}}

\newcommand{\len}{\operatorname{length}}

\newcommand{\grade}{\operatorname{grade}}

\newcommand{\ext}{\operatorname{Ext}}	
	
\newcommand{\lotimes}{\otimes^{\mathbf{L}}}
\newcommand{\HH}{\operatorname{H}}
\newcommand{\Hom}{\operatorname{Hom}}	
\newcommand{\coker}{\operatorname{Coker}}

\newcommand{\tor}{\operatorname{Tor}}
\newcommand{\im}{\operatorname{Im}}

\newcommand{\Ker}{\operatorname{Ker}}

\newcommand{\ideal}[1]{\mathfrak{#1}}
\newcommand{\m}{\ideal{m}}
\newcommand{\n}{\ideal{n}}
\newcommand{\p}{\ideal{p}}

\newcommand{\fa}{\ideal{a}}

\newcommand{\fr}{\ideal{r}}

\newcommand{\comp}[1]{\widehat{#1}}
\newcommand{\ol}{\overline}

\newcommand{\xra}{\xrightarrow}

\newcommand{\vf}{\varphi}

\newcommand{\x}{\mathbf{x}}

\renewcommand{\geq}{\geqslant}
\renewcommand{\leq}{\leqslant}
\renewcommand{\ker}{\Ker}
\renewcommand{\hom}{\Hom}

\newcommand{\fm}{\mathfrak{m}}
\newcommand{\fn}{\mathfrak{n}}
\newcommand{\om}{\omega_{R}}
\newcommand{\Ext}{\operatorname{Ext}}
\newcommand{\Tor}{\operatorname{Tor}}

\numberwithin{equation}{lem}

\begin{document}

\bibliographystyle{amsplain}

\title
{Rings that are homologically of minimal multiplicity}

\author{Keivan Borna}

\address{Keivan Borna: Faculty of Mathematical Sciences and Computer, Tarbiat Moallem University, Tehran,
Iran and School of Mathematics, Institute for
Research in Fundamental Sciences (IPM), P.O. Box: 19395-5746, Tehran, Iran.}

\email{borna@ipm.ir}

\urladdr{http://www.ipm.ac.ir/personalinfo.jsp?PeopleCode=IP0500002}

\thanks{K. Borna was supported in part by a grant from
IPM (No. 88130035).}

\author{Sean Sather-Wagstaff}

\address{Sean Sather-Wagstaff:
Department of Mathematics,
NDSU Dept \# 2750,
PO Box 6050,
Fargo, ND 58108-6050
USA.}

\email{Sean.Sather-Wagstaff@ndsu.edu}

\urladdr{http://www.ndsu.edu/pubweb/\~{}ssatherw/}

\thanks{S. Sather-Wagstaff was supported in part by a grant from
the NSA}

\author{Siamak Yassemi}

\address{Siamak Yassemi: Department of Mathematics, University of
Tehran, Tehran, Iran and School of Mathematics, Institute for
Research in Fundamental Sciences (IPM), Tehran, Iran.}

\email{yassemi@ipm.ir}

\urladdr{http://math.ipm.ac.ir/yassemi}

\thanks{S. Yassemi was supported in part by a grant from
IPM (No. 88130214).}

\keywords{Betti numbers, canonical module, Gorenstein rings,
minimal multiplicity.}

\subjclass[2000]{13D07, 13D02, 13H10}

\date{\today}

\begin{abstract}
Let $R$ be a local Cohen-Macaulay ring with canonical module
$\om$. We investigate the following question of Huneke: If
the sequence of Betti numbers $\{\beta_i^R(\om)\}$ has polynomial growth,
must $R$ be  Gorenstein? This question is well-understood when $R$ has
minimal multiplicity. We investigate this question for a 
more general class of rings which we say are homologically
of  minimal multiplicity.
We provide several characterizations of the rings in this class and
establish a general ascent and descent result.
\end{abstract}

\maketitle

\section{Introduction} \label{sec0}

Throughout this paper $(R,\m,k)$ is a commutative local noetherian ring. 
Recall that a finitely generated $R$-module $\om$ is a \emph{canonical
module} for $R$ if 
$$\ext^i_R(k,\om)\cong\begin{cases}
k & \text{if $i=\dim(R)$} \\ 0 & \text{if $i\neq\dim(R)$.}
\end{cases}$$
In some of the literature, canonical modules are also called dualizing modules.
They were introduced by Grothendieck~\cite{hartshorne:lc}
for the study of local cohomology.
Foxby~\cite{foxby:gmarm}, 
Reiten~\cite{reiten:ctsgm} and 
Sharp~\cite{sharp:gmccmlr} prove that 
$R$ admits a canonical module if and only if $R$ is Cohen-Macaulay
and a homomorphic image of a local Gorenstein ring. In particular, if $R$ is
complete and Cohen-Macaulay, then it admits a canonical module. 

One useful property is the following:
The ring $R$ is Gorenstein if and only if $R$ is its own canonical module.
This leads to  the following question of Huneke.\footnote{To the best of our knowledge,
Huneke has only posed this question in conversations and talks, not in print.}

\begin{question} \label{q01}
Assume that $R$ is Cohen-Macaulay with canonical module $\om$.
If the sequence of Betti numbers $\{\beta_i^R(\om)\}$ is bounded above by
a polynomial in $i$, must $R$ be Gorenstein?
\end{question}

For rings of minimal multiplicity, it is straightforward to answer 
this question: reduce to the case where
$\m^2=0$ and show that $\beta_i^R(\om)=(r^2-1)r^{i-1}$ for all $i\geq 1$;
here $r$ is the Cohen-Macaulay type of $R$.
(See Example~\ref{ex11} below.) 
This question has been answered in the affirmative 
for other classes of rings by
Jorgensen and Leuschke~\cite{jorgensen:gbscm}
and Christensen, Striuli and Veliche~\cite{christensen:gmirlr}.
These classes include the classes of Golod rings, rings with codimension at most 3, rings that are
one link from a complete intersection, rings with $\m^3=0$, 
Teter rings, and nontrivial fiber product rings.

In this paper, we investigate Question~\ref{q01} for the following classes of rings
which contain the rings of minimal multiplicity: 

\begin{defn}\label{defn01}
Let $m$, $n$ and $t$ be integers with $m,t\geq 1$ and $n\geq 0$.
The ring $R$ is \emph{homologically of minimal multiplicity 
of type $(m,n,t)$}
if there exists a 
local ring homomorphism 
$\vf\colon(R,\fm,k)\to (S,\fn,l)$ and a finitely generated $S$-module $M\neq 0$
such that 
\begin{enumerate}[\quad(1)]
\item\label{defn01a}
the ring $S$ has a canonical module $\omega_{S}$,
\item\label{defn01b}
the map $\vf$ is flat with Gorenstein closed fibre $S/\m S$,
\item\label{defn01c}
one has 
$\tor^S_i(\omega_S,M)=0$ for $i\geq t$, and
\item \label{defn01d}
one has $\n^2M=0$ and $m=\beta^S_0(M)$ and $n=\beta^S_0(\n M)$.
\end{enumerate}
The ring $R$ is \emph{strongly homologically of minimal multiplicity 
of type $(m,n)$}
if there exists a 
local ring homomorphism 
$\vf\colon(R,\fm,k)\to (S,\fn,l)$ and a finitely generated $S$-module $M\neq 0$
satisfying conditions~\eqref{defn01a}, \eqref{defn01b}, \eqref{defn01d}, and
\begin{enumerate}[\quad(3')]
\item\label{defn01'c}
the $S$-module $M$ is in the Auslander class $\cata(S)$.
\end{enumerate}
(Consult Section~\ref{sec4} for background information on Auslander classes.)
\end{defn}

The following facts are proved in Section~\ref{sec4}.
If $R$ is Cohen-Macaulay and has minimal multiplicity, then it is
strongly homologically of minimal multiplicity of type $(1,e(R)-1)$; here $e(R)$ is the 
Hilbert-Samuel multiplicity of $R$ with respect to $\m$. 
If $R$ is Gorenstein, then it is
strongly homologically of minimal multiplicity of type $(m,n)$ for all integers $m,n\geq 1$.
Also, if $R$ is homologically of minimal multiplicity, then it is Cohen-Macaulay.

We provide an affirmative answer to Question~\ref{q01} for 
rings that are strongly  homologically of minimal multiplicity
in
the following result, which is contained in Theorems~\ref{para11} and~\ref{para11s}.

\begin{thm}\label{thm01} 
Assume that $R$ is homologically of minimal multiplicity of type 
$(m,n,t)$ and with canonical module $\om$.
\begin{enumerate}[\quad\rm(a)]
\item\label{thm01a} 
One has
$\beta_{t+s}^{R}(\om)=(n/m)^{s}\cdot \beta_{t}^{R}(\om)$
for all $s\geq 0$.
\item\label{thm01b}  
If $n > m$ and $R$ is not Gorenstein, then the sequence $\{\beta_{i}^{R}(\om)\}$ grows
exponentially.
\item\label{thm01c}  
If $n = m$, then the sequence $\{\beta_{i}^{R}(\om)\}$ is eventually
constant.
\item\label{thm01d}  
If $n < m$, then $R$ is Gorenstein.
\item\label{thm01e}  
If $R$ is not Gorenstein, then $m\mid n$.
\item\label{thm01f}  
If $R$ is strongly homologically of minimal multiplicity of type 
$(m,n)$ and
$n = m$, then $R$ is Gorenstein.
\end{enumerate}
\end{thm}

Section~\ref{sec1} also contains further analysis of the behavior of the 
Betti numbers under various hypotheses.
While this investigation is motivated by questions about the Betti numbers 
of  canonical modules, our
methods yield results about  Betti numbers
of arbitrary modules. 
For instance, Theorem~\ref{thm01}\eqref{thm01f} is essentially a special case of
Theorem~\ref{prop13}\eqref{prop13d}. 
Accordingly, we state and prove these more general results,
and periodically give explicit specializations to the case of 
rings that are (strongly) homologically of minimal multiplicity. 

Section~\ref{sec3} contains three alternate characterizations of the rings
that are homologically of minimal multiplicity. One of them,
Theorem~\ref{disc11d}, states that, 
if $R$ is homologically of minimal multiplicity, then one can assume in Definition~\ref{defn01}
that the homomorphism $\vf$ is flat with regular closed fibre and that the ring $S$
is complete with algebraically closed residue field.
The second, Theorem~\ref{thm51} shows that $R$ is homologically of minimal
multiplicity whenever there is a ``quasi-Gorenstein'' homomorphism $R\to S$
satisfying conditions~\eqref{defn01a}, \eqref{defn01c} and~\eqref{defn01d}
of Definition~\ref{defn01}. (Definition~\ref{defn51} contains 
background information on quasi-Gorenstein homomorphisms.)
The third characterization is dual to the original definition, using Ext-vanishing
in place of Tor-vanishing; see Remark~\ref{disc13} and Proposition~\ref{disc11c}.
Similar characterizations are given for rings that are 
strongly homologically of minimal multiplicity.

Finally, Section~\ref{sec9} documents ascent and descent behavior for these classes of rings.
The most general statements are contained in Corollaries~\ref{prop41}
and~\ref{prop41x}.
The result for flat maps is given here; see Theorems~\ref{lem43} and~\ref{lem43x}.

\begin{thm} \label{thm02}
Assume that $\psi\colon R\to R'$ is a flat local ring
homomorphism with Gorenstein closed fibre $R'/\m R'$. 
If  $R'$
is (strongly) homologically of minimal multiplicity,
then so is $R$. The converse holds when 
$k$ is perfect and $R'/\m R'$ is regular.
\end{thm}

Example~\ref{ex9454} shows that the converse statement can  fail
when $R'/\m R'$ is only assumed to be of minimal multiplicity.
It also shows that, in general, the localized tensor product of 
rings that are strongly homologically of minimal multiplicity
need not be homologically of minimal multiplicity.
On the other hand, we do not know at this time whether this 
class of rings is closed under localization. See Section~\ref{sec9}
for other open problems.

\section{Basic Properties}\label{sec4}

In this section we make some observations about
rings that are (strongly) homologically of minimal multiplicity.
We  begin with a definition that is due to Foxby.

\begin{defn}\label{defn11}
Let $S$ be a Cohen-Macaulay local ring with canonical module $\omega_S$.
The \emph{Auslander class} of $S$ is the class $\cata(S)$ consisting of all
$R$-modules $M$ 
satisfying the following conditions:
\begin{enumerate}[\quad(1)]
\item
the natural map $\xi_M\colon M\to\Hom_S(\omega_S,\omega_S\otimes_SM)$
given by $\xi_M(m)(x)=x\otimes m$ is an isomorphism, and
\item
one has
$\tor^S_i(\omega_S,M)=0=\ext^i_S(\omega_S,\omega_S\otimes_SM)$ for all $i\geq 1$.
\end{enumerate}
The \emph{Bass class} of $S$ is the class $\catb(S)$ consisting of all
$R$-modules $M$ 
satisfying the following conditions:
\begin{enumerate}[\quad(1)]
\item
the natural map $\gamma_M\colon \omega_S \otimes_S\Hom_S(\omega_S ,M)\to M$
given by $\gamma_M(x\otimes\psi)=\psi(x)$ is an isomorphism, and
\item
one has
$\ext_S^i(\omega_S,M)=0=\tor_i^S(\omega_S,\Hom_S(\omega_S,M))$ 
for all $i\geq 1$.
\end{enumerate}
\end{defn}

Here are some straightforward facts about Auslander classes.

\begin{disc}\label{disc11}
Let $S$ be a Cohen-Macaulay local ring with canonical module $\omega_S$.
The Auslander class $\cata(S)$
contains every projective $S$-module. Furthermore, if two modules in a 
short exact sequence are in $\cata(S)$, then so is the third module.
It follows that $\cata(S)$
contains every $S$-module of finite projective dimension.

From the definitions, we conclude that rings that are
strongly homologically of minimal multiplicity of type $(m,n)$
are homologically of minimal multiplicity of type $(m,n,1)$.
Also, with $\vf$ and $M$ as in Definition~\ref{defn01},
the condition $n\geq 1$ implies that $\n M\neq 0$.
\end{disc}

For the sake of clarity, we recall the definition of minimal multiplicity,
first studied by Abhyankar~\cite{abhyankar:lrhed}.

\begin{defn}\label{defn12}
Let $(R,\m)$ be a local ring. 
The \emph{Hilbert-Samuel multiplicity} of $R$, denoted $e(R)$, is the normalized leading
coefficient of the polynomial that agrees with the function $\len_R(R/\m^{n})$ for $n\gg 0$.
If $R$ is Cohen-Macaulay, then there is an inequality $e(R)\geq \beta^R_0(\m)-\dim(R)+1$,
and $R$ has \emph{minimal multiplicity} when $e(R)= \beta^R_0(\m)-\dim(R)+1$.
\end{defn}

\begin{ex} \label{ex11}
Let $k$ be a field, let $r$ be a positive integer, and consider the ring
$R=k[X_1,\ldots,X_r]/(X_1,\ldots,X_r)^2$. This is a local
artinian ring of minimal multiplicity, with multiplicity $e(R)=r+1$
and type $r$.
(In particular $R$ is Gorenstein if and only if $r=1$.)
Hence, the canonical module $\omega_R$ has
$\beta_0^R(\omega_R)=r$. Furthermore, the exact sequence
$$0\to k^{r^2-1}\to R^r\to\om\to 0$$
(obtained by truncating a minimal free resolution of $\om$)
can be used to show that
$\beta_i^R(\om)=(r^2-1)r^{i-1}$ for all $i\geq 1$.
\end{ex}

We will have several opportunities to use the following fact
from~\cite[0.(10.3.1)]{grothendieck:ega3-1}.

\begin{disc} \label{disc15}
Let $(R,\m,k)$ be a local ring and let $\vf_0\colon k\to l$ be a field extension.
Then there is a flat local ring homomorphism
$\vf\colon (R,\m,k)\to (S,\n,l)$ such that $S$ is complete, the 
extension $k\to l$ induced by $\vf$ is precisely $\vf_0$,
and $\n=\m S$.
\end{disc}

The next three
results explain the location of rings
homologically of minimal multiplicity in the heierarchy of rings.

\begin{prop} \label{disc11a}
If $R$ is a local Cohen-Macaulay ring with minimal multiplicity, then it is
strongly homologically of minimal multiplicity of type $(1,e(R)-1)$. 
\end{prop}

\begin{proof}
Remark~\ref{disc15} provides a flat local ring homomorphism 
$\psi\colon(R,\m,k) \to (S,\n,l)$
such that $S$ is complete,
$l$ is the algebraic closure of $k$
and $\n=\m S$.
It follows readily that $S$ is Cohen-Macaulay and 
has a canonical module $\omega_S$. Furthermore, we have
$e(S)=e(R)$ and $\beta^S_0(\n)= \beta^R_0(\m)$
and $\dim(S)=\dim(R)$. In particular, the ring $S$ has minimal multiplicity.

The fact that $S$ is Cohen-Macaulay and has infinite residue field
implies that 
there exists an $S$-regular sequence $\x\in\n\smallsetminus\n^2$ such that
$\len_S(S/(\x)S)=e(S)$.
(The sequence $\x$ generates a minimal reduction of $\n$.)
This explains the second equality in the next sequence:
\begin{align*}
\beta^S_0(\n)-\dim(S)+1
&=e(S) \\
&=\len_S(S/(\x)S)\\
&= 1+\beta^S_0(\fn/(\x)S)+\len_S(\n^2 (S/(\x)S))\\
&= 1+\beta^S_0(\n)-\dim(S)+\len_S(\n^2 (S/(\x)S)).
\end{align*}
The first equality is from the minimal multiplicity condition.
The third equality is explained by the filtration
$\n^2 (S/(\x)S)\subseteq\n (S/(\x)S) \subseteq S/(\x)S$.
The fourth equality is from the fact that
$\x$ is a maximal $S$-regular sequence in $\n\smallsetminus\n^2$.
From this sequence, it follows 
that $\n^2(S/(\x)S)=0$. (See also the proof of~\cite[(1)]{abhyankar:lrhed}.)

Since the sequence $\x$ is $S$-regular, 
the $S$-module $M=S/(\x)S$ has finite projective dimension.
Remark~\ref{disc11} then implies that $M\in\cata(S)$.
It follows that $R$ is strongly homologically of minimal multiplicity of type
$(m,n)$ where
$m=\beta^S_0(M)=1$ and $n=\beta^S_0(\n M)=e(R)-1$.
\end{proof}

\begin{prop} \label{disc11z}
If $R$ is a local Gorenstein ring, then it is
strongly homologically of minimal multiplicity of type $(m,n)$
for all integers $m\geq 1$ and $n\geq 0$.
\end{prop}

\begin{proof}
Fix integers $m\geq 1$ and $n\geq 0$.
The ring $S=R[\![X_1,\ldots,X_n]\!]$
is local with maximal ideal $\n=(\m,X_1,\ldots,X_n)S$
and residue field $k$. The natural inclusion
$\vf\colon R\to S$ is flat with Gorenstein closed fibre
$S/\m S\cong k[\![X_1,\ldots,X_n]\!]$.
Since $R$ is Gorenstein, the same is true of $S$.
Thus $S$ has canonical module $\omega_S=S$.
It follows readily from the definition that every $S$-module is in
$\cata(S)$. 
In particular, the $S$-module
$$M=k^{m-1}\oplus S/(\m S+((X_1,\ldots,X_n)S)^2)$$ 
is in $\cata(S)$. 
It is straightforward to show that 
$\n^2 M=0$ and
$\beta^S_0(M)=m$.
To complete the proof, use  the following isomorphisms
\begin{align*}
\n M
&\cong \n S/(\m S+((X_1,\ldots,X_n)S)^2) \\
&\cong (X_1,\ldots,X_n)k[\![X_1,\ldots,X_n]\!]/((X_1,\ldots,X_n)k[\![X_1,\ldots,X_n]\!])^2\\
&\cong k^n
\end{align*}
to see that $\beta^S_0(\n M)=n$.
\end{proof}

\begin{prop}\label{disc11b}
If $R$ is homologically of minimal multiplicity, then it is Cohen-Macaulay.
\end{prop}

\begin{proof}
Definition~\ref{defn01}, 
the ring $S$ has a canonical module, so it is Cohen-Macaulay.
The homomorphism $\vf$ is flat and local, and it follows that $R$ is Cohen-Macaulay.
\end{proof}

\begin{disc} \label{d67zzz}
If $R$ is strongly homologically of minimal multiplicity,
then $R$ need not have minimal multiplicity. To see this, 
let $R$ be a local Gorenstein ring that is not of minimal multiplicity.
(For example, it is straightforward to show that the ring $R=k[X]/(X^3)$ satisfies
these conditions.) Proposition~\ref{disc11z} shows that
$R$ is strongly homologically of minimal multiplicity.
\end{disc}

\begin{disc} \label{d67}
The following diagram summarizes the relations between the classes of rings
under consideration:
$$\xymatrix{
\txt{Cohen-Macaulay and  minimal multiplicity} \ar@<1ex>@{=>}[d]^-{\eqref{disc11a}} \\
\txt{strongly homologically  of minimal multiplicity}
\ar@<1ex>@{=>}[d]^-{\eqref{disc11}} 
\ar@<1ex>@{=>}[u] |-{-} ^-{\eqref{d67zzz}} 
\ar@{=>}@<1ex>[r]|-{|}^-{\eqref{ex11}+\eqref{disc11a}}
& \text{Gorenstein} \ar@<1ex>@{=>}[l]^-{\eqref{disc11z}} \\
\txt{homologically of  minimal multiplicity}\ar@<1ex>@{=>}[r]^-{\eqref{disc11b}} 
\ar@<1ex>@{==>}[u]^{\eqref{qblah}} 
&
\txt{Cohen-Macaulay.}\ar@<1ex>@{=>}[l] |-{|}^-{\eqref{ex9454}}
}$$
At this time we do not know 
whether the vertical implication marked~\eqref{qblah} holds.
We pose this explicitly as a question next.
\end{disc}

\begin{question} \label{qblah}
If $R$ is homologically  of minimal multiplicity, must
it be strongly homologically  of minimal multiplicity?
\end{question}

We end this section with a natural result to be used later.

\begin{lem}\label{lem41}
Let $\fa$ be an ideal of $R$ with
$\fa$-adic completion $\comp{R}^{\fa}$.
\begin{enumerate}[\rm(a)]
\item\label{lem41a}
Then $R$ is homologically of minimal multiplicity of type $(m,n,t)$
if and only if  $\comp R^{\fa}$
is  homologically of minimal multiplicity of type $(m,n,t)$.
\item\label{lem41b}
Then $R$ is strongly homologically of minimal multiplicity of type $(m,n)$
if and only if $\comp R^{\fa}$
is strongly homologically of minimal multiplicity of type $(m,n)$.
\end{enumerate}
\end{lem}

\begin{proof}
We prove part~\eqref{lem41a}; the proof of part~\eqref{lem41b} is similar.

Assume that $\comp {R}^{\fa}$ is homologically of minimal multiplicity of type $(m,n,t)$.
Let $\vf_1\colon \comp R^{\fa}\to S_1$ be a ring homomorphism, and let $M_1$ be an $S_1$-module
as in Definition~\ref{defn01}. It is straightforward to verify that the composition
$\vf_1\psi\colon R\to S_1$ and the $S_1$-module $M_1$ satisfy the axioms to show that 
$R$ is homologically of minimal multiplicity of type $(m,n,t)$.

Assume next that $R$ is homologically of minimal multiplicity of type $(m,n,t)$.
Let $\vf_2\colon R\to S_2$ be a ring homomorphism, and let $M_2$ be an $S_2$-module
as in Definition~\ref{defn01}. Then the induced map
$\comp {\vf_2}^{\fa}\colon \comp R^{\fa}\to \comp {S_2}^{\fa}$ and the 
$\comp{S_2}^{\fa}$ module $M_2\cong\comp{M_2}^{\fa}$ 
show that 
$\comp R^{\fa}$ is homologically of minimal multiplicity of type $(m,n,t)$.
\end{proof}



\section{Patterns in Betti Numbers} \label{sec1}

This section contains the proof of Theorem~\ref{thm01} from the introduction.

\begin{thm}\label{thm11} 
Let $(S,\n,l)$ be a local  ring, and let $M$ and $N$ be 
finitely generated $S$-modules.
Let $m$, $n$ and $t$ be integers, and
assume that there is an exact sequence of $S$-module homomorphisms
$0\to l^{n}\to M\to l^{m}\to 0$.
\begin{enumerate}[\quad\rm(a)]
\item\label{thm11a} 
If $\Tor_{i}^{S}(N,M)=0$
for $i=t,t+1$, then 
$m\beta_{t+1}^{S}(N)=n\beta_{t}^{S}(N)$. 
\item\label{thm11b} 
If $m\geq 1$ and
$\Tor_{i}^{S}(N,M)=0$ for $i=t,\ldots,t+s$ for some positive integer $s$, then
$\beta_{t+s}^{S}(N)=(n/m)^{s}\cdot \beta_{t}^{S}(N)$.
\end{enumerate}
\end{thm}

\begin{proof}
For each integer $i$, we have
$\Tor_i^S(N,l^m)\cong l^{m\beta^S_i(N)}$
and
$\Tor_i^S(N,l^n)\cong l^{n\beta^S_i(N)}$.
Thus, a piece of the long exact sequence in $\Tor^S(N,-)$ associated to the given sequence
has the form
\[\Tor_{i+1}^{S}(N,M)\to
l^{m\beta^S_{i+1}(N)}\to l^{n\beta^S_i(N)}\to \Tor_{i}^{S}(N,M).
\]
If  $\Tor_{t}^{S}(N,M)=0=\Tor_{t+1}^{S}(N,M)$, then
the sequence yields an isomorphism
$l^{m\beta^S_{t+1}(N)}\cong l^{n\beta^S_t(N)}$
and hence the equality 
$m\beta_{t+1}^{S}(N)=n\beta_{t}^{S}(N)$.
Recursively, if $\Tor_{i}^{S}(N,M)=0$ for $i=t,\ldots,t+s$, then
$\beta_{t+s}^{S}(N)=(n/m)^{s}\cdot \beta_{t}^{S}(N)$.
\end{proof}

The next result is  dual to the previous one.
It can be proved using the ideas from Theorem~\ref{thm11}
with $\ext^i_S(N,-)$ in place of $\tor^S_i(N,-)$. 
See also Remark~\ref{disc13}.

\begin{thm}\label{thm11'} 
Let $(S,\n,l)$ be a local  ring, and let $M$ and $N$ be 
finitely generated $S$-modules.
Let $m$, $n$ and $t$ be integers, and
assume that there is an exact sequence of $S$-module homomorphisms
$0\to l^{m}\to M\to l^{n}\to 0$.
\begin{enumerate}[\quad\rm(a)]
\item\label{thm11'a} 
If $\Ext^{i}_{S}(N,M)=0$
for $i=t,t+1$, then 
$n\beta_{t}^{S}(N)=m\beta_{t+1}^{S}(N)$. 
\item\label{thm11'b} 
If $m\geq 1$ and
$\Ext^{i}_{S}(N,M)=0$ for $i=t,\ldots,t+s$ for some positive integer $s$, then
$\beta_{t+s}^{S}(N)=(n/m)^{s}\cdot \beta_{t}^{S}(N)$.
\qed
\end{enumerate}
\end{thm}

\begin{disc}\label{disc13}
Theorem~\ref{thm11'} is more than just dual to Theorem~\ref{thm11}; it  is equivalent to
Theorem~\ref{thm11}. To show this, we require a few facts from Matlis duality. 
Let $(S,\n,l)$ be a local ring.
Let $E_S(l)$ denote the injective hull of the residue field $l$,
and consider the Matlis duality functor $(-)^\vee=\Hom_S(-,E_S(l))$. 

First, recall that a module $M$ has finite length if and only if its Matlis dual $M^\vee$
has finite length. 
When $M$ has finite length, the natural biduality map $M\to M^{\vee\vee}$ is an
isomorphism, and $\n^2M=0$ if and only if $\n^2M^\vee=0$.
Using this, it is straightforward to show that there is
an exact sequence $0\to l^{m}\to M\to l^{n}\to 0$
if and only if there is an exact sequence
$0\to l^{n}\to M^\vee\to l^{m}\to 0$.

Second, if $M$ has finite length, then there  are isomorphisms
$$\Tor^S_i(N,M^\vee)^\vee\cong \ext^i_S(N,M^{\vee\vee})\cong\ext^i_S(N,M).$$
for each integer $i$ and each $S$-module $N$.
The first is a version of Hom-tensor adjointness, and the second comes from the
biduality isomorphism $M\to M^{\vee\vee}$.
Thus, we have $\Tor^S_i(N,M^\vee)=0$ if and only if $\ext^i_S(N,M)=0$.
(Furthermore, it is readily shown that
$M\in\cata(S)$ if and only if $M^{\vee}\in\catb(S)$.)
The equivalence of Theorems~\ref{thm11} and~\ref{thm11'}  now follows readily.

Each of the remaining results of this paper has a dual version that is equivalent
via a similar argument.
Because the results are equivalent, and not just similar, 
we only state the ``Tor-version'', and leave the 
``Ext-version'' for the reader.
\end{disc}

\begin{cor}\label{cor11} 
Let $(S,\n,l)$ be a local  ring, and let $M$ and $N$ be 
finitely generated $S$-modules.
Let $m$, $n$ and $t$ be integers with $m\geq 1$, and
assume that there is an exact sequence of $S$-module homomorphisms
$0\to l^{n}\to M\to l^{m}\to 0$.
Assume that $\Tor_{i}^{S}(N,M)=0$ for all $i\geq t$.
\begin{enumerate}[\quad\rm(a)]
\item \label{cor11a} 
If $n > m$ and $\beta_{t}^{S}(N)\neq 0$, then the sequence $\{\beta_{i}^{S}(N)\}$ grows
exponentially.
\item \label{cor11b} 
If $n = m$, then the sequence $\{\beta_{i}^{S}(N)\}$ is eventually
constant.
\item \label{cor11c} 
If $n < m$, then the sequence $\{\beta_{i}^{S}(N)\}$
is eventually zero, that is, the module $N$ has finite projective dimension.
\item \label{cor11d} 
If $N$ has infinite projective dimension, then $\n M\neq 0$ and
the number $n/m$ is a positive integer.
\end{enumerate}
\end{cor}

\begin{proof}
Theorem~\ref{thm11}\eqref{thm11b} 
implies that $\beta_{i}^{S}(N)=(n/m)^{i-t}\cdot
\beta_{t}^{S}(N)$ for all $i\geq t$. The conclusions~\eqref{cor11a}--\eqref{cor11c} now follow
immediately in this case, recalling that $N$ has finite projective dimension
if and only if $\beta_i^S(N)=0$ for $i\gg 0$. 

For part~\eqref{cor11d}, assume that $N$ has infinite projective dimension. 
It follows that 
$\beta_{i}^{S}(N)=(n/m)^{i-t}\cdot
\beta_{t}^{S}(N)$ is a positive integer for all $i\geq t$.
We conclude that $n/m$ is a positive integer.
If $\n M=0$, then $M\cong k^{m+n}$. Since $m+n\geq 1$, our Tor-vanishing
assumption implies that $\Tor_{t}^{S}(N,k)=0$, contradicting
the infinitude of $\pd_S(N)$.
\end{proof}

The next result contains parts~\eqref{thm01a}--\eqref{thm01e} 
of Theorem~\ref{thm01} from the introduction.

\begin{thm}\label{para11} 
Assume that $R$ is homologically of minimal multiplicity of type 
$(m,n,t)$, and set $r=n/m$.
\begin{enumerate}[\quad\rm(a)]
\item\label{para11a} 
If $R$ has a  canonical module $\om$, then
$\beta_{t+s}^{R}(\om)=r^{s}\cdot \beta_{t}^{R}(\om)$
for all $s\geq 0$.
\item\label{para11b}  
Assume that $n > m$ and $R$ has a  canonical module $\om$.
If $R$ is not Gorenstein, then the sequence $\{\beta_{i}^{R}(\om)\}$ grows
exponentially.
\item\label{para11c}  
If $n = m$ and $R$ has a  canonical module $\om$, 
then the sequence $\{\beta_{i}^{R}(\om)\}$ is eventually
constant.
\item\label{para11d}  
If $n < m$, then $R$ is Gorenstein.
\item\label{para11e}  
If $R$ is not Gorenstein, then $m\mid n$ and $n\geq 1$.
\end{enumerate}
\end{thm}

\begin{proof} 
Using Lemma~\ref{lem41}\eqref{lem41a} we may assume that $R$ is complete,
so $R$ has a  canonical module $\om$ in~\eqref{para11d}--\eqref{para11e}.
Let $\vf\colon R\to S$ be as in Definition~\ref{defn01}.
The fact that $\vf$ is flat with Gorenstein closed fibre
implies that $\omega_S\cong S\otimes_R\om$ and
$\tor^R_i(S,\om)=0$ for all $i\geq 1$. It follows that
$\beta^R_i(\om)=\beta^S_i(\omega_S)$ for all $i$.
The desired conclusions now follow from
Theorem~\ref{thm11}\eqref{thm11b} and Corollary~\ref{cor11},
using the fact that $R$ is Gorenstein if and only if $\beta^R_i(\omega_R)=0$
for some $i\geq 1$, equivalently, for all $i\gg 0$.
\end{proof}

The following question is motivated by Theorem~\ref{para11}\eqref{para11e}.

\begin{question}\label{q11}
Assume that $R$ is not Gorenstein.
If $R$  is 
homologically of minimal multiplicity
of type $(r,rm,t)$,
must $R$ be
homologically of minimal multiplicity
of type $(1,m,t)$? 
If $R$  is 
strongly homologically of minimal multiplicity
of type $(r,rm)$,
must $R$ be
strongly homologically of minimal multiplicity
of type $(1,m)$? 
\end{question}

The next result gives two criteria that yield affirmative answers for Question~\ref{q11}

\begin{prop} \label{prop11}
Let $(S,\n,l)$ be a local  ring, and let $M$ and $N$ be 
finitely generated $S$-modules.
Let $m$, $n$ and $t$ be integers with $m\geq 1$, and
assume that there is an exact sequence of $S$-module homomorphisms
\begin{equation}\label{prop11e}
0\to l^{n}\to M\xra\tau l^{m}\to 0
\end{equation}
and that
 $\Tor_{i}^{S}(N,M)=0$
for $i\geq t$.
Assume that $\pd_S(N)$ is infinite, and set $r=n/m$ and
$e=\edim(S)=\beta^S_0(\n)$.
\begin{enumerate}[\quad\rm(a)]
\item \label{prop11a}
There is an equality
$\beta^S_0(M)=m$.
\item \label{prop11d}
There are inequalities
$r\leq \len_S(S/\ann_S(M))-1\leq e$.
\item \label{prop11b}
One has
$r=\len_S(S/\ann_S(M))-1$ if and only if $M\cong (S/\ann_S(M))^m$.
\item \label{prop11c}
One has
$r=e$ if and only if $M\cong (S/\n^2)^m$. 
\end{enumerate}
\end{prop}

\begin{proof}
Set $J=\ann_S(M)$ and $a=\len_S(S/J)$.

\eqref{prop11a}
The surjection $\tau\colon M\twoheadrightarrow l^m$
implies that $\beta^S_0(M)\geq m$. Suppose that $\beta^S_0(M)>m$.
It follows that $\ker(\tau)\cong l^n$ contains a minimal generator for $M$.
We conclude that $M\cong l\oplus M'$ for some submodule $M'\subseteq M$.
(To see this, let $x_1\in M$ be a minimal generator in $l^n$, and complete this
to a minimal generating sequence $x_1,\ldots,x_p$ for $M$. The module
$M/(x_2,\ldots,x_p)$ is cyclic and nonzero, generated by the residue of $x_1$,
which we denote $\ol{x_1}$.
Since $\n x_1=0$ it follows that $M/(x_2,\ldots,x_p)\cong l\ol{x_1}$.
The composition $lx_1\subseteq M\to M/(x_2,\ldots,x_p)\cong l\ol{x_1}$
is an isomorphism, so the surjection $M\to l$ splits.)
The condition $0=\tor_i^S(N,M)\cong\tor_i^S(N,M')\oplus \tor_i^S(N,l)$ for $i\geq t$
implies that $\tor_t^S(N,l)=0$,
contradicting the infinitude of $\pd_S(N)$.

\eqref{prop11d}
Since $\n^2M=0$, we have $\n^2\subseteq J\subseteq\n$
and hence
$$a-1\leq\len(S/\n^2)-1=e.$$
This is the second desired inequality.

There is an $S$-module epimorphism
$\pi\colon(S/J)^m\twoheadrightarrow M$. Since $m=\beta^S_0(M)$,
we conclude that $\ker(\pi)\subseteq\n (S/J)=\n/J$.
Since $\n^2\subseteq J$, we see that $\n\ker(\pi)=0$,
so $\ker(\pi)\cong l^s$ for some integer $s$.
Using the exact sequence
$$0\to l^s\to(S/J)^m\xra{\pi} M\to 0$$
we have the first equality in the next sequence
\begin{align*}
s
&=am-\len_S(M)
=am-(m+n)
=m(a-1-r).
\end{align*}
The second equality is from the sequence~\eqref{prop11e}.
The third equality is from the definition $r=n/m$.
Since $s\geq 0$ and $m> 0$, we have $a-1-r\geq 0$ that is,
$r\leq a-1$. This completes the proof of~\eqref{prop11d}.

For the rest of the proof, 
we continue with the notation from the proof of part~\eqref{prop11d}. 

\eqref{prop11b}
We have $M\cong (S/J)^m$ if and only if $\pi$ is an isomorphism,
that is, if and only if $l^s\cong \ker(\pi)=0$.
Since $s=m(a-1-r)$ and $m>0$, we conclude that $s=0$ if and only if $r=a-1$.

\eqref{prop11c}
Assume first that $r=e$.  Part~\eqref{prop11d} implies that
$r\leq a-1\leq e=r$ and thus $r=a-1$. Hence, part~\eqref{prop11b} 
yields an isomorphism
$M\cong (S/J)^m$. The surjection $S/\n^2\twoheadrightarrow S/J$
yields the  inequality in the next sequence
$$a=\len_S(S/J)\leq\len_S(S/\n^2)=e+1=r+1=a.$$
It follows that $\len_S(S/J)=\len_S(S/\n^2)$, so the surjection
$S/\n^2\twoheadrightarrow S/J$ is an isomorphism.
Hence, we have $J=\n^2$, and thus $M\cong (S/J)^m\cong (S/\n^2)^m$. 

For the converse, assume that $M\cong (S/\n^2)^m$. 
It follows that, in the exact sequence~\eqref{prop11e} we have
$l^{rm}\cong l^n\cong\n M\cong (\n/\n^2)^m\cong l^{em}$
and hence $r=e$.
\end{proof}

The next  results describe relations between
$m$, $n$, $\beta^S_1(N)$ and $\beta^S_0(N)$.

\begin{prop} \label{prop12}
Let $(S,\n,l)$ be a local  ring, and let $M$ and $N$ be 
finitely generated $S$-modules such that $\pd_S(N)$ is infinite.
Let $m$ and $n$ be integers with $m\geq 1$, and
assume that there is an exact sequence of $S$-module homomorphisms
\begin{equation}\label{prop12e}
0\to l^{n}\xra\alpha M\xra\tau l^{m}\to 0
\end{equation}
and that
 $\Tor_{i}^{S}(N,M)=0$
for $i\geq 1$.
Set $r=n/m$.
\begin{enumerate}[\quad\rm(a)]
\item \label{prop12a}
There is an inequality
$\beta^S_1(N)\leq \beta^S_0(N)r$.
\item \label{prop12d}
There is an equality
$\ker(N\otimes_S\tau)=\n(N\otimes_SM)$.
\item \label{prop12b}
One has $\beta^S_1(N)= \beta^S_0(N)r$ if and only if $\n(N\otimes_SM)=0$.
\end{enumerate}
\end{prop}

\begin{proof}
For each index $i$, set $b_i=\beta^S_i(N)$.

\eqref{prop12a}
Apply $N\otimes_S-$ to the sequence~\eqref{prop12e} to obtain the following exact sequence
\begin{equation}
\label{eq2}
0\to\tor^S_1(N,l)^m\xra\gamma N\otimes_S l^{mr}\xra{N\otimes_S\alpha} 
N\otimes_SM\xra{N\otimes_S\tau} N\otimes_S l^m\to 0.
\end{equation}
Notice that we have
$$\tor^S_1(N,l)^m\cong (l^{b_1})^m\cong l^{b_1m}
\qquad\text{and}\qquad
N\otimes_S l^{mr}\cong l^{b_0mr}.
$$
The sequence~\eqref{eq2} implies that $\tor^S_1(N,l)^m\subseteq
N\otimes_Sl^{mr}$, so we have
$b_1m\leq b_0mr$. Since $m\geq 1$, this implies $b_1\leq b_0r$.

\eqref{prop12d}
Proposition~\ref{prop11}\eqref{prop11a} shows that $m=\beta^S_0(M)$ and
moreover, the surjection $\tau$ is naturally identified with the 
natural surjection $M\to M\otimes_Sl$.
Accordingly, we have $l^n\cong\n M$, so the sequence~\eqref{prop12e} has the form
$$0\to\n M\xra\alpha M\xra\tau l^m\to 0.$$
Thus, the sequence~\eqref{eq2} has the form
$$0\to\tor^S_1(N,l)^m\xra\gamma N\otimes_S\n M\xra{N\otimes_S\alpha} 
N\otimes_SM\xra{N\otimes_S\tau} N\otimes_S l^m\to 0.
$$
It follows that
$\n(N\otimes_SM)=\im(N\otimes_S\alpha)=\ker(N\otimes_S\tau)$.

\eqref{prop12b}
We have
$$\n(N\otimes_SM)=\ker(N\otimes_S\tau)\cong\coker(\gamma)\cong
l^{m(b_0r-b_1)}.$$
Hence, we have $\n(N\otimes_SM)=0$
if and only if $m(b_0r-b_1)=0$, that is, if and only if $b_1= b_0r$.
\end{proof}

\begin{cor} \label{cor12}
Let $R$ be a local ring with a canonical module $\om$.
If $R$ is homologically of minimal multiplicity of type $(m,n,1)$,
then $\beta^R_1(\om)\leq\beta^R_0(\om)n/m$.
\end{cor}

\begin{proof}
If $R$ is Gorenstein, then $\beta^R_1(\om)=0\leq\beta^R_0(\om)n/m$.
When $R$ is not Gorenstein,
argue as  in the proof of Theorem~\ref{para11} to derive the desired inequality
from Proposition~\ref{prop12}\eqref{prop12a}.
\end{proof}

Note that the hypotheses of parts~\eqref{prop13a} and~\eqref{prop13d} 
of the next result
hold automatically when $N=\omega_S$ and $M$ is in the Auslander class $\mathcal{A}(S)$.

\begin{thm} \label{prop13}
Let $(S,\n,l)$ be a local  ring, and let $M$ and $N$ be 
finitely generated $S$-modules such that $\pd_S(N)$ is infinite.
Let $m$ and $n$ be integers with $m\geq 1$, and
assume that there is an exact sequence of $S$-module homomorphisms
\begin{equation}\label{prop13e}
0\to l^{n}\to M\xra\tau l^{m}\to 0
\end{equation}
and that
 $\Tor_{i}^{S}(N,M)=0$
for $i\geq 1$.
Set $r=n/m$.
\begin{enumerate}[\quad\rm(a)]
\item \label{prop13a}
If $M\cong\Hom_S(N,N\otimes_SM)$, then $\beta^S_1(N)<\beta^S_0(N)r$.
\item \label{prop13d}
Assume that $\ext^1_S(N,N\otimes_SM)=0$
and $\len_S(\Hom_S(N,N\otimes_SM))=\len_S(M)$. Then there are equalities 
\begin{align*}
\beta^S_1(N)&=\textstyle\frac{1}{2}
\left[\beta^S_0(N)(r+1)\pm\sqrt{\beta^S_0(N)^2(r+1)^2-4(\beta^S_0(N)^2-1)(r+1)}\right]\\
&=\textstyle\frac{1}{2}
\left[\beta^S_0(N)(r+1)\pm\sqrt{(r+1)[\beta^S_0(N)^2(r-3)+4]}\right].
\end{align*}
In particular, the integer 
$$\beta^S_0(N)^2(r+1)^2-4(\beta^S_0(N)^2-1)(r+1)
=(r+1)[\beta^S_0(N)^2(r-3)+4]$$ 
is a perfect square.
\item \label{prop13b}
If $N=\omega_S\ncong S$ and $M\in\mathcal{A}(S)$, then 
$$\beta^S_1(\omega_S)=\beta^S_0(\omega_S)(r^2-1)/r 
\qquad\text{and}\qquad
r=\frac{\beta^S_1(\omega_S)+\sqrt{\beta^S_1(\omega_S)^2+4\beta^S_0(\omega_S)^2}}
{2\beta^S_0(\omega_S)}.$$
\end{enumerate}
\end{thm}

\begin{proof}
For each index $i$, set $b_i=\beta^S_i(N)$.
Note that Corollary~\ref{cor11}\eqref{cor11c} implies
that $r\geq 1$ and $\n M\neq 0$.

\eqref{prop13a}
Since we have $\n M\neq 0$, the isomorphism
$M\cong\Hom_S(N,N\otimes_SM)$
implies that $\n(N\otimes_SM)\neq 0$. The conclusion 
$b_1<b_0r$ follows from parts~\eqref{prop12a} and~\eqref{prop12b}
of Proposition~\ref{prop12}.

\eqref{prop13d}
By definition, we have $N\otimes_Sl^m\cong l^{b_0m}$.
We have seen that $\ker(N\otimes_S\tau)\cong l^{m(b_0r-b_1)}$,
so the sequence~\eqref{eq2} yields the next exact sequence:
$$0\to l^{m(b_0r-b_1)}\to N\otimes_SM\xra{N\otimes_S\tau}l^{b_0m}\to 0.$$
Our Ext-vanishing assumption implies that the associated long exact sequence
in $\ext_S(N,-)$ begins as follows
\begin{align*}
0
\to\Hom_S(N,l)^{m(b_0r-b_1)}
&\to\Hom_S(N,N\otimes_S M)\\
&\to\Hom_S(N, l)^{b_0m}
\to\ext^1_S(N,l)^{m(b_0r-b_1)}
\to 0.
\end{align*}
Using the standard isomorphism
$\ext^i_S(N,l)\cong l^{b_i}$, we conclude that
this sequence has the following form:
$$
0\to l^{b_0m(b_0r-b_1)}
\to\Hom_S(N,N\otimes_S M)
\to l^{b_0^2m}
\to l^{b_1m(b_0r-b_1)}
\to 0.$$
Thus, our length assumption explains the second equality in the next sequence
\begin{align*}
b_1m(b_0r-b_1)
&=b_0^2m-\len_S(\Hom_S(N,N\otimes_S M))+b_0m(b_0r-b_1)\\
&=b_0^2m-\len_S(M)+b_0m(b_0r-b_1)\\
&=b_0^2m-m(r+1)+b_0m(b_0r-b_1).
\end{align*}
Dividing by $m$ and simplifying, we find that
$$b_1^2-b_0(r+1)b_1+(b_0^2-1)(r+1)=0.$$
The desired conclusions now follow from the quadratic formula.

\eqref{prop13b}
Employ the notation of the exact sequence~\eqref{eq2}.
We have shown that
$$\ker(\omega_S\otimes_S\tau)\cong\coker(\gamma)\cong l^{m(b_0r-b_1)}$$
so the exact sequence~\eqref{eq2} provides the next exact sequence
$$0\to l^{m(b_0r-b_1)}\to \omega_S\otimes_SM\to l^{mb_0}\to 0.$$
Furthermore, the condition $M\in \mathcal{A}(S)$ implies that
$\ext^i_S(\omega_S,\omega_S\otimes_SM)=0$ for all $i\geq 1$.
We conclude from Theorem~\ref{thm11'}\eqref{thm11'b} that
$$b_i=b_1\left(\frac{mb_0}{m(b_0r-b_1)}\right)^{i-1}=b_1\left(\frac{b_0}{b_0r-b_1}\right)^{i-1}$$
for all $i\geq 1$. On the other hand, we know that
$b_i=r^{i-1}b_1$ for all $i\geq 1$. Since we are assuming that $b_i\neq 0$ for all $i$,
we conclude that $r=b_0/(b_0r-b_1)$.
Solve this equation for $b_1$ to derive the first desired equality.
For the second equality, 
substitute $b_1=b_0(r^2-1)/r$ into the expression
$\frac{b_1+\sqrt{b_1^2+4b_0^2}}{2b_0}$ and simplify.
\end{proof}

\begin{cor} \label{prop13z}
Assume that $R$ is strongly homologically of minimal multiplicity of type 
$(m,n)$ and with canonical module $\om\ncong R$.
If $r=n/m$, then
$$\beta^R_1(\omega_R)=\beta^R_0(\omega_R)(r^2-1)/r 
\qquad\text{and}\qquad
r=\frac{\beta^R_1(\omega_R)+\sqrt{\beta^R_1(\omega_R)^2+4\beta^R_0(\omega_R)^2}}
{2\beta^R_0(\omega_R)}.$$
\end{cor}

\begin{proof}
This follows directly from Theorem~\ref{prop13}\eqref{prop13b}
because $\beta^R_i(\omega_R)=\beta^S_i(\omega_S)$.
\end{proof}

The following example is from~\cite[(3.4)]{gasharov:bpclr}.
It demonstrates how our results can yield exact values for the Betti numbers of
canonical modules. 
It also  shows that,
if $R$  is 
strongly homologically of minimal multiplicity
with $S$ and $M$ as in Definition~\ref{defn01},
then $M$ may not be a direct sum of cyclic $S$-modules.
Similar arguments yield the 
Betti numbers of the canonical modules for the rings constructed 
in~\cite{avramov:pmivpd}.

\begin{ex} \label{blah}
Let $k$ be a field and let $\alpha\in k$ such that $\alpha\neq 0,1,-1$.
Consider the polynomial ring $A=k[X_1,X_2,X_3,X_4]$ and the ideal 
$I\subseteq A$ generated by the following polynomials:
$$\alpha X_1X_3+X_2X_3, \quad
X_1X_4+X_2X_4, \quad
X_3^2, \quad
X_4^2, \quad
X_1^2, \quad
X_2^2, \quad
X_3X_4.
$$
The ring $R=A/I$ is artinian and local with maximal ideal
$\m=(x_1,x_2,x_3,x_4)R$,
and $\m^3=0$. (Here $x_i$ denotes the image of $X_i$ in $R$.)
For each integer $n$, set 
$$d_n=\begin{pmatrix}x_1 & \alpha^nx_3+x_4 \\ 0 & x_2\end{pmatrix}.$$
Consider the following chain complex of $R$-modules
$$G=\cdots\xra{d_{n+1}} R^2\xra{d_{n}} R^2\xra{d_{n-1}}\cdots$$
and the $R$-module $M=\im(d_0)$. Let $\om$ denote a canonical module for $R$.

Arguing as in~\cite[(3.1)]{gasharov:bpclr}, one has the following facts.
The complexes $G$ and $\hom_R(G,R)$ are exact.
In the language of~\cite[(4.1.2)]{christensen:gd},
this means that $G$ is a ``complete resolution'' of $M$ by finite free modules.
Using~\cite[(4.1.3),(4.2.6),(4.4.13)]{christensen:gd}, we conclude that
$M\in\cata(R)$. Also, one has
$\m^2M=0$ and $\beta^R_0(M)=2$ and $\beta^R_0(\m M)=6$,
so the ring $R$ is strongly homologically of minimal multiplicity
of type $(6,2)$. In particular $\len_R(M)=8$, and the complex
$$G'=\cdots\xra{d_{2}} R^2\xra{d_{1}} R^2\to 0$$
is a minimal free resolution of $M$.
The gist of~\cite[(3.4)]{gasharov:bpclr} is that
\begin{equation} \label{blah1}
\ker(d_{n+2})\not\cong \ker(d_n)
\end{equation}
for all $n\geq 1$.

The socle of $R$ is $\m^2$, which has basis $x_1x_2, x_1x_3,x_1x_4$.
Hence, we have $\beta^R_0(\om)=3$. 
Corollary~\ref{prop13z} implies that
$\beta^R_1(\om)=8$, and Theorem~\ref{para11}\eqref{para11a} yields the formula
$\beta^R_n(\om)=8\cdot 3^{n-1}$ for all $n\geq 1$.
\footnote{Preliminary computations were performed using Macaulay 2~\cite{M2}.}

We claim that $M$ is indecomposable. 
By way of contradiction,
suppose that $M\cong M_1\oplus M_2$ where $M_1$ and $M_2$ are both nonzero.
The equality $\beta^R_0(M)=2$ implies that each $M_i$ is cyclic.
It follows that $\pd_R(M_i)=\infty$ for $i=1,2$. Indeed, if $\pd_R(M_i)$ is finite, then
the fact that $R$ is artinian implies that $M_i$ is free. Since $M_i$ is cyclic, we have
$M_i\cong R$, and so 
$$8=\len_R(M)=\len_R(M_1)+\len_R(M_2)>\len_R(M_i)=8$$
which is impossible. 

The resolution $G'$ shows that $\beta_n^R(M)=2$ for all $n\geq 0$.
It follows that $\beta_n^R(M_i)=1$ for all $n\geq 0$ and for $i=1,2$. 
Let $F_i$ be the minimal free resolution of $M_i$ with $n$th differential $d_{i,n}$.
From~\cite[(3.8)]{gasharov:bpclr}, it follows that there is an integer $n\geq 1$ such that
$\ker(d_{i,n+2})\cong\ker(d_{i,n})$ for $i=1,2$. 
The uniqueness of minimal free resolutions implies that
$G'\cong F_1\oplus F_2$, and hence
$$\ker(d_{n+2})\cong \ker(d_{1,n+2})\oplus\ker(d_{2,n+2})\cong \ker(d_{1,n})\oplus\ker(d_{2,n})
\cong \ker(d_n).$$
This contradicts~\eqref{blah1}. Thus $M$ is indecomposable, as claimed.
\end{ex}

The next result contains 
Theorem~\ref{thm01}\eqref{thm01f} from the introduction.

\begin{thm}\label{para11s} 
Assume that $R$ is strongly homologically of minimal multiplicity of type 
$(m,n)$.
If $n = m$, then $R$ is Gorenstein.
\end{thm}

\begin{proof}
Using Lemma~\ref{lem41}\eqref{lem41b}, we assume that $R$ is complete.
Hence $R$ has a canonical module $\om$.
The assumption $m=n$ translates as $r=1$, so
Corollary~\ref{prop13z} implies that $\beta^R_1(\om)=0$.
It follows that $R$ is Gorenstein.
\end{proof}

The following question asks if the conclusion of Theorem~\ref{para11s}
holds when $R$ is only assumed to be 
homologically of minimal multiplicity.

\begin{question}
Assume that $R$ is homologically of minimal multiplicity of type 
$(m,n,t)$.
If $n = m$, must $R$ be Gorenstein?
\end{question}

\section{Alternate Characterizations} \label{sec3}

In this section, we provide alternate characterizations of the 
rings that are (strongly) homomologically of minimal multiplicity.
The first of these results is Theorem~\ref{disc11d} which says that
in the definition of ``homologically of minimal multiplicity'' one can assume
that the ring $S$ is complete with algebraically closed residue field
and that the closed fibre $S/\m S$ is regular.
In preparation, we recall some background information on
local ring homomorphisms.

\begin{defn} \label{defn41}
Let $\vf\colon R\to S$ be a local ring homomorphism.
A \emph{Cohen factorization} of $\vf$ is a diagram of local ring homomorphisms
$R\xra{\dot\vf}R'\xra{\vf'}S$
satisfying the following conditions:
\begin{enumerate}[\quad(1)]
\item
one has $\vf=\vf'\dot\vf$,
\item 
the map $\dot\vf$ is flat with regular closed fibre $R'/\m R'$,
\item 
the local ring $R'$ is complete, and
\item
the map $\vf'$ is surjective.
\end{enumerate}
\end{defn}

\begin{disc} \label{disc41}
Let $\vf\colon R\to S$ be a local ring homomorphism.
If $\vf$ admits a Cohen factorization, then the ring $S$ is a homomorphic image
of a complete local ring, so it is complete. Conversely, if $S$ is complete,
then $\vf$ admits a Cohen factorization by~\cite[(1.1)]{avramov:solh}.
\end{disc}

\begin{defn} \label{defn41'}
Let $(R,\m,k)$ be a local ring. The
\emph{$i$th Bass number} of $R$ is the integer
$\mu^i_R(R)=\rank_k(\ext^i_R(k,R))$.

Let 
$\vf\colon (R,\m)\to (S,\n)$ be a local ring homomorphism,
and assume that $R$ is Cohen-Macaulay. 
The homomorphism $\vf$ has 
\emph{finite  flat dimension} 
if $S$ has finite flat dimension as an $R$-module, that is,
if $S$ admits a bounded resolution by flat $R$-modules.
The homomorphism $\vf$ is \emph{Gorenstein}
if it has finite flat dimension and 
$\mu^{i+\depth(S)}_S(S)=\mu^{i+\depth(R)}_R(R)$ for all $i$.
An ideal $I\subset R$ is \emph{Gorenstein} if the natural
surjection $R\to R/I$ is Gorenstein.
\end{defn}

\begin{disc} \label{disc41'}
Let $\vf\colon (R,\m,k)\to(S,\n,l)$ be a local ring homomorphism.
If $\vf$ is flat, then it has finite flat dimension.
Also, when $\vf$ is flat, it is Gorenstein if and only if the closed fibre $S/\m S$ is Gorenstein;
see~\cite[(4.2)]{avramov:lgh}.
An ideal $I\subset R$ generated by an $R$-regular sequence is Gorenstein.
An ideal $I\subset R$ is Gorenstein if and only if
the $R$-module $R/I$ is perfect and $\beta^{R}_g(R/I)=1$ where $g=\grade_{R}(R/I)=\pd_{R}(R/I)$;
see~\cite[(4.3)]{avramov:lgh}.
If $\vf$ has finite flat dimension and $\psi\colon S\to T$ is another local
homomorphism of finite flat dimension, then~\cite[(4.6)]{avramov:lgh}
implies that the composition $\psi\vf$ is Gorenstein if and only if
$\psi$ and $\vf$ are both Gorenstein.

Assume that $\vf$ admits a Cohen factorization
$R\xra{\dot\vf}R'\xra{\vf'}S$.
The map $\vf$ has finite flat dimension if and only if
$\pd_{R'}(S)$ is finite; see~\cite[(3.2)]{avramov:solh}.
The map $\vf$ is Gorenstein if and only if
$\ker(\vf')$ is a Gorenstein ideal of $R'$;
see~\cite[(3.11)]{avramov:solh}.

Assume that $\vf$ is Gorenstein and
that $S$ is Cohen-Macaulay.
Since $\pd_{R'}(S)$ is finite, it follows that $R'$ is Cohen-Macaulay.
Since $R'$ and $S$ are both complete, they each admit a canonical module,
and~\cite[(5.7)]{christensen:scatac} implies that
$\omega_S\cong S\otimes_{R'}\omega_{R'}$
and $\tor^{R'}_i(S,\omega_{R'})=0$ for all $i\geq 1$.
\end{disc}

\begin{thm}\label{disc11d}
A local  ring  $R$ is homologically of minimal multiplicity
of type $(m,n,t)$
if and only if 
there exists a 
local ring homomorphism 
$\vf\colon(R,\fm,k)\to (S,\fn,l)$ and a finitely generated $S$-module $M\neq 0$
such that 
\begin{enumerate}[\quad\rm(1)]
\item\label{disc11d1}
the ring $S$ is complete and Cohen-Macaulay
with canonical module $\omega_S$, and $l$ is algebraically closed,
\item\label{disc11d2}
the map $\vf$ is flat with regular closed fibre $S/\m S$,
\item\label{disc11d3}
one has 
$\tor^S_i(\omega_S,M)=0$ for $i\geq t$, and
\item \label{disc11d4}
one has $\n^2M=0$ and $m=\beta^S_0(M)$ and $n=\beta^S_0(\n M)$.
\end{enumerate}
\end{thm}

\begin{proof}
One implication is routine. For the converse, assume that 
$R$ is homologically of minimal multiplicity of type $(m,n,t)$.
We complete the proof in three steps.

Step 1: By definition, there is a 
local ring homomorphism 
$\vf_1\colon(R,\fm,k)\to (S_1,\fn_1,l_1)$ and a finitely generated $S_1$-module $M_1\neq 0$
such that 
\begin{enumerate}[\quad($1'$)]
\item
the ring $S_1$ has a canonical module $\omega_{S_1}$,
\item
the map $\vf_1$ is flat with Gorenstein closed fibre $S_1/\m S_1$,
\item
one has 
$\tor^{S_1}_i(\omega_{S_1},M_1)=0$ for $i\geq t$, and
\item 
one has $\n_1^2M_1=0$ and $m=\beta^{S_1}_0(M_1)$ and $n=\beta^{S_1}_0(\n_1 M_1)$.
\end{enumerate}

Step 2: 
From Remark~\ref{disc15}, there is a flat local homomorphism 
$\psi\colon(S_1,\n_1,l_1) \to (S_2,\n_1 S_2,l)$
such that $S_2$ is complete and $l$ is the algebraic closure of $l_1$.
Since the map $\psi$ is flat and the maximal ideal of $S_2$ is $\n_2=\n_1 S_2$,
it is straightforward to show that the composition 
$\vf_2\colon R\xra{\vf_1} S_1\xra\psi S_2$ and the module $M=S_2\otimes_{S_1}M_1$
satisfy the following conditions:
\begin{enumerate}[\quad($1''$)]
\item
the ring $S_2$ is complete and Cohen-Macaulay with canonical module
$\omega_{S_2}$ and has an algebraically closed residue field,
\item
the map $\vf_2$ is flat with Gorenstein closed fibre $S_2/\m S_2$,
\item
one has 
$\tor^{S_2}_i(\omega_{S_2},M)=0$ for $i\geq t$, and
\item 
one has $\n_2^2M=0$ and $m=\beta^{S_2}_0(M)$ and $n=\beta^{S_2}_0(\n_2 M)$.
\end{enumerate}

Step 3: The ring $S_2$ is complete, so the local homomorphism
$\vf_2$ admits a Cohen factorization 
$(R,\m,k)\xra{\vf}(S,\n,l)\xra{\vf_2'}(S_2,\n_2,l)$.
Remark~\ref{disc41'} implies the following:
the ideal $\ker(\vf_2')\subset S$ is Gorenstein,
the ring $S$ is complete and Cohen-Macaulay,
there is an isomorphism
$\omega_{S_2}\cong S_2\otimes_{S}\omega_{S}$,
and  $\tor^{S}_i(S_2,\omega_{S})=0$ for all $i\geq 1$.

Let $F$ be a free resolution of $\omega_{S}$ over $S$.
Then $F\otimes_{S}S_2$ is a free resolution of 
$S_2\otimes_{S}\omega_{S}\cong\omega_{S_2}$.
Hence, for each index $i$ there are isomorphisms
\begin{align}
\label{disc11d3a}
\tor_i^{S}(\omega_{S},M)
&\cong\HH_i(F\otimes_{S}M)
\cong\HH_i((F\otimes_{S}S_2)\otimes_{S_2}M)
\cong\tor^{S_2}_i(\omega_{S_2},M).
\end{align}
It follows that the map 
$\vf$ and the module $M$
satisfy the conditions~\eqref{disc11d1}--\eqref{disc11d4}.
\end{proof}

The next result is a version of Theorem~\ref{disc11d} for rings
that are strongly homologically of minimal multiplicity;
it is proved similarly.

\begin{thm}\label{disc11dx}
A local  ring  $R$ is strongly homologically of minimal multiplicity
of type $(m,n)$
if and only if 
there exists a 
local ring homomorphism 
$\vf\colon(R,\fm,k)\to (S,\fn,l)$ and a finitely generated $S$-module $M\neq 0$
such that 
\begin{enumerate}[\quad\rm(1)]
\item\label{disc11dx1}
the ring $S$ is complete and Cohen-Macaulay, and $l$ is algebraically closed,
\item\label{disc11dx2}
the map $\vf$ is flat with regular closed fibre $S/\m S$,
\item\label{disc11dx3}
one has 
$M\in\cata(S)$, and
\item \label{disc11dx4}
one has $\n^2M=0$ and $m=\beta^S_0(M)$ and $n=\beta^S_0(\n M)$.
\qed
\end{enumerate}
\end{thm}

Readers familiar with~\cite{avramov:rhafgd} will recognize that
the proof of Theorem~\ref{thm01} only requires the homomorphism
$\vf$ to be \emph{quasi}-Gorenstein. 
(See Definition~\ref{defn51}.)
One may ask why  we require the stronger hypotheses
in Definition~\ref{defn01}.
Theorem~\ref{thm51} shows that our definition is equivalent
to the weaker definition which only requires $\vf$ to be quasi-Gorenstein.
We have chosen this one since 
flat maps with Gorenstein closed fibres are more familiar.

\begin{defn} \label{defn51}
Let 
$\vf\colon (R,\m)\to (S,\n)$ be a local ring homomorphism,
and assume that $R$ is Cohen-Macaulay. 
The homomorphism $\vf$ has 
\emph{finite G-dimension} if the $\n$-adic completion $\comp S$
is in the Auslander class $\cata(\comp R)$ of the $\m$-adic completion $\comp R$.
The homomorphism $\vf$ is \emph{quasi-Gorenstein}
if it has finite G-dimension and 
$\mu^{i+\depth(S)}_S(S)=\mu^{i+\depth(R)}_R(R)$ for all $i$.
An ideal $I\subset R$ is \emph{quasi-Gorenstein} if the natural
surjection $R\to R/I$ is quasi-Gorenstein.
\end{defn}

\begin{disc} \label{disc51}
Let 
$\vf\colon (R,\m)\to (S,\n)$ be a local ring homomorphism,
and assume that $R$ is Cohen-Macaulay. 
Let $\grave\vf\colon R\to\comp S$ denote the composition of $\vf$
with the natural map $S\to \comp S$.
Fix a Cohen factorization
$R\xra{\dot\vf} R'\xra{\vf'} \comp S$ of  $\grave\vf$.
Since $R$ is Cohen-Macaulay and $\dot\vf$ is flat with
regular closed fibre, the ring $R'$ is Cohen-Macaulay.
As $R'$ is complete, it has a canonical module $\omega_{R'}$.

The homomorphism $\vf$ has finite G-dimension if and only if
$\comp S\in\cata(R')$; see~\cite[(4.1.7) and (4.3)]{avramov:rhafgd}.
In particular, if $\vf$ has finite G-dimension, then
$\tor^{R'}_i(\comp S,\omega_{R'})=0$ for $i\geq 1$. 
Moreover, if $\vf$ is flat (or more generally, if $\vf$ has finite flat dimension)
then $\vf$ has finite G-dimension.
If $\vf$ is Gorenstein (e.g., if it is flat with Gorenstein closed fibre)
then it is quasi-Gorenstein.
The composition of two quasi-Gorenstein homomorphisms is quasi-Gorenstein
by~\cite[(8.9)]{avramov:rhafgd}.

If $\vf$ is quasi-Gorenstein, then $S$ is Cohen-Macaulay, and
the canonical module of $\comp S$ is $\omega_{\comp S}
\cong\comp S\otimes_{R'}\omega_{R'}$. Indeed, 
from~\cite[(7.8)]{avramov:rhafgd} we conclude that the complex
$\comp S\lotimes_{R'}\omega_{R'}$ is 
a dualizing complex for $\comp S$.
(See~\cite{avramov:rhafgd} for an extensive discussion on
the topic of dualizing complexes.)
The vanishing $\tor^{R'}_i(\comp S,\omega_{R'})=0$ for $i\geq 1$
implies that $\comp S\lotimes_{R'}\omega_{R'}$
is isomorphic (in the derived category $\catd(\comp S)$)
to the module $\comp S\otimes_{R'}\omega_{R'}$.
It follows that this is a canonical module for $\comp S$,
and thus $S$ is
Cohen-Macaulay.
\end{disc}

\begin{thm}\label{thm51}
A local  ring  $R$ is homologically of minimal multiplicity
of type $(m,n,t)$
if and only if 
it is Cohen-Macaulay and
there exists a 
local ring homomorphism 
$\vf\colon(R,\fm)\to (S,\fn)$ and a finitely generated $S$-module $M\neq 0$
such that 
\begin{enumerate}[\quad\rm(1)]
\item\label{thm51a}
the ring $S$ has a canonical module $\omega_{S}$,
\item
the map $\vf$ is quasi-Gorenstein,
\item\label{thm51c}
one has 
$\tor^S_i(\omega_S,M)=0$ for $i\geq t$, and
\item \label{thm51d}
one has $\n^2M=0$ and $m=\beta^S_0(M)$ and $n=\beta^S_0(\n M)$.
\end{enumerate}
\end{thm}

\begin{proof}
One implication is routine,
using the fact that a local homomorphism that is flat with Gorenstein
closed fibre is quasi-Gorenstein. For the converse, assume that 
$R$ 
is Cohen-Macaulay and
there exists a 
local ring homomorphism 
$\vf\colon R\to S$ and a finitely generated $S$-module $M\neq 0$
satisfying conditions~\eqref{thm51a}--\eqref{thm51d}.
By passing to the completion $\comp S$, we assume that $S$ is complete.

Fix a Cohen factorization
$R\xra{\dot\vf} R'\xra{\vf'} S$ of $\vf$.
Remark~\ref{disc51} implies that $R'$ is Cohen-Macaulay
with canonical module $\omega_{R'}$,
that $\tor^{R'}_i(\comp S,\omega_{R'})=0$ for $i\geq 1$, and 
that the canonical module of $\comp S$ is $\omega_{\comp S}
\cong\comp S\otimes_{R'}\omega_{R'}$.
The argument of Theorem~\ref{disc11d} now shows that
the homomorphism $\dot\vf$ and the $R'$-module $M$
satisfy the hypotheses of Definition~\ref{defn01},
 so $R$ is homologically of minimal multiplicity
 of type $(m,n,t)$.
\end{proof}

The next result is a version of Theorem~\ref{thm51} for rings
that are strongly homologically of minimal multiplicity;
it is proved similarly.

\begin{thm}\label{thm51x}
A local  ring  $R$ is homologically of minimal multiplicity
of type $(m,n)$
if and only if 
it is Cohen-Macaulay and
there exists a 
local ring homomorphism 
$\vf\colon(R,\fm)\to (S,\fn)$ and a finitely generated $S$-module $M\neq 0$
such that 
\begin{enumerate}[\quad\rm(1)]
\item\label{thm51xa}
the ring $S$ has a canonical module $\omega_{S}$,
\item
the map $\vf$ is quasi-Gorenstein,
\item\label{thm51xc}
one has 
$M\in\cata(S)$, and
\item \label{thm51xd}
one has $\n^2M=0$ and $m=\beta^S_0(M)$ and $n=\beta^S_0(\n M)$.
\end{enumerate}
\end{thm}

The final results of this section explain why we do not single out rings that satisfy the 
conditions that are dual to
``(strongly) homologically of minimal multiplicity''.

\begin{prop}\label{disc11c}
A local  ring  $R$ is homologically of minimal multiplicity
if and only if there exists a 
local ring homomorphism 
$\vf\colon(R,\fm,k)\to (S,\fn,l)$ and a finitely generated $S$-module $N\neq 0$
such that 
\begin{enumerate}[\rm\quad(1)]
\item
the ring $S$ has a canonical module $\omega_{S}$,
\item
the map $\vf$ is flat with Gorenstein closed fibre $S/\m S$,
\item 
one has $\ext_S^i(\omega_S,N)=0$ for $i\geq t$, and
\item
one has $\n^2N=0$.
\end{enumerate}
\end{prop}

\begin{proof}
By Remark~\ref{disc13} 
an $S$-module $N$ satisfies
$\n^2N=0$ if and only if $\n^2N^\vee=0$,
and $\ext_S^i(\omega_S,N)=0$ if and only if $\tor^S_i(\omega_S,N^\vee)=0$.
The result now follows.
\end{proof}

\begin{prop}\label{disc11cx}
A local  ring  $R$ is homologically of minimal multiplicity
if and only if there exists a 
local ring homomorphism 
$\vf\colon(R,\fm,k)\to (S,\fn,l)$ and a finitely generated $S$-module $N\neq 0$
such that 
\begin{enumerate}[\rm\quad(1)]
\item
the ring $S$ has a canonical module $\omega_{S}$,
\item
the map $\vf$ is flat with Gorenstein closed fibre $S/\m S$, and
\item
one has
$N\in\catb(S)$, and
\item
one has $\n^2N=0$.
\end{enumerate}
\end{prop}

\begin{proof}
The proof is similar to that of Proposition~\ref{disc11c}.
\end{proof}

\section{Ascent and Descent Behavior} \label{sec9}

This section culminates in
Corollaries~\ref{prop41}
and~\ref{prop41x} which describe ascent and descent behavior for our classes of rings
along local quasi-Gorenstein ring homomorphisms. We divide the proofs into several pieces.

\begin{lem}\label{lem42}
Let $I\subset R$ be a quasi-Gorenstein ideal. 
If the quotient $R/I$
is homologically of minimal multiplicity of type $(m,n,t)$,
then $R$ is homologically of minimal multiplicity of type $(m,n,t)$.
The converse holds when $I\subseteq \m^2$.
\end{lem}

\begin{proof}
Let $\tau\colon R\to R/I$ 
denote the canonical surjection.

Assume first that 
the quotient $R/I$
is homologically of minimal multiplicity of type $(m,n,t)$.
Let $\vf_1\colon (R/I,\m/I)\to (S_1,\n_1)$ be a  ring homomorphism and $M_1$ an $S_1$-module
as in Theorem~\ref{disc11d}. 
Since $S_1$ is complete, Remark~\ref{disc41} implies that the composition
$\vf_1\tau\colon R\to S_1$ has a Cohen factorization
$R\xra{\dot\vf_1}R'\xra{\vf_1'}S_1$.
Since $\tau$ and $\vf_1$ are quasi-Gorenstein,
Remark~\ref{disc51} implies that the composition $\vf_1'\dot\vf_1=\vf_1\tau$
is quasi-Gorenstein. 
Hence, we have
$\omega_{S_1}\cong S_1\otimes_{R'}\omega_{R'}$
and $\tor^{R'}_i(S_1,\omega_{R'})=0$ for all $i\geq 1$.
The isomorphisms~\eqref{disc11d3a} from the proof of Theorem~\ref{disc11d} shows that 
\begin{align*}
\tor_i^{R'}(\omega_{R'},M_1)
\cong\tor^{S_1}_i(\omega_{S_1},M_1)=0
\end{align*}
for all $i\geq t$.
Let $\m'$ be the maximal ideal of $R'$.
Since $(\m')^2M_1=\n_1^2M_1=0$, the homomorphism
$\dot\vf_1\colon R\to R'$ and the $R'$-module $M_1$
combine to show that $R$ is homologically of minimal multiplicity
of type $(m,n,t)$.

For the converse, assume that $I\subseteq\m^2$ and that
$R$ is homologically of minimal multiplicity of type $(m,n,t)$.
Let $\vf_2\colon (R,\m)\to (S_2,\n_2)$ be a ring homomorphism and $M_2$ an $S_2$-module
as in Definition~\ref{defn01}. Since $\vf_2$ is flat,
it is straightforward to show that the ideal
$IS_2\subseteq S_2$ is quasi-Gorenstein 
(see, e.g., \cite[(8.6)]{avramov:rhafgd})
and the induced homomorphism
$\ol{\vf_2}\colon R/I\to S_2/IS_2$ is flat. The closed fibre
the composition
$\ol{\vf_2}\tau=\pi_2\vf_2\colon R\to S_2/IS_2$ is 
the ring $(S_2/IS_2)\otimes_{R}(R/\m)\cong S_2/\m S_2$ 
which is Gorenstein.
The assumptions $I\subseteq\m^2$ and $\n_2^2M_2=0$ imply that
$IS_2M_2=0$, so $M_2$ is naturally an $S_2/IS_2$-module. 
As in the previous paragraph, we have
$\omega_{S_2/IS_2}\cong S_2/IS_2\otimes_{S_2}\omega_{S_2}$
and $\tor^{S_2}_i(S_2/IS_2,\omega_{S_2})=0$ for all $i\geq 1$,
and 
$$\tor^{S_1}_i(\omega_{S_1},M_2)\cong\tor_i^{S_2/IS_2}(\omega_{S_2/IS_2},M_2)=0$$
for all $i\geq t$.
Since $(\n_2/IS_2)^2M_2=\n_2^2M_2=0$, the homomorphism
$\ol{\vf_2}\colon R/I\to S_2/IS_2$ and the $S_2/IS_2$-module $M_2$
combine to show that $R/I$ is homologically of minimal multiplicity.
\end{proof}

\begin{lem}\label{lem42x}
Let $I\subset R$ be a quasi-Gorenstein ideal. 
If the quotient $R/I$
is strongly homologically of minimal multiplicity of type $(m,n)$,
then $R$ is strongly homologically of minimal multiplicity of type $(m,n)$.
The converse holds when $I\subseteq \m^2$.
\end{lem}

\begin{proof}
This is proved like Lemma~\ref{lem42}.
\end{proof}

The next question asks if the converses in Lemmas~\ref{lem42}
and~\ref{lem42x} hold without the assumption $I\subseteq\m^2$.

\begin{question}
Let $I\subset R$ be a quasi-Gorenstein ideal. 
If  $R$
is (strongly) homologically of minimal multiplicity,
must $R/I$ be (strongly) homologically of minimal multiplicity?
\end{question}

Before continuing toward our general results on ascent and descent,
we note a few special cases of Lemmas~\ref{lem42} and~\ref{lem42x}.

\begin{ex} \label{ex41}
If $(R,\m)$ is a local Cohen-Macaulay ring of 
minimal multiplicity and $I\subseteq \m^2$ is a quasi-Gorenstein ideal,
then $R/I$ is strongly homologically of minimal multiplicity;
see Proposition~\ref{disc11a}.

If $(S,\n)$ is a Cohen-Macaulay local ring and $\x\in\n^2$ is an $S$-regular sequence,
then $S$ is (strongly) homologically of minimal multiplicity if and only if
$S/(\x)S$ is (strongly) homologically of minimal multiplicity.
If $S$ has minimal multiplicity and $(\x)S\neq 0$, then
$S/(\x)S$ is strongly homologically of minimal multiplicity,
but is not of minimal multiplicity.
\end{ex}

Example~\ref{ex41} gives a method for constructing rings that
are strongly homologically of minimal multiplicity. The next question asks whether this is 
essentially the only way. In other words, it asks whether there is a structure
theorem for rings that are strongly homologically of minimal multiplicity
akin to Cohen's structure theorem, where regular rings are replaced by rings 
of minimal multiplicity.

\begin{question}
If $R$ is strongly homologically of minimal multiplicity,
must there be an isomorphism $\comp R\cong Q/I$
where $Q$ is a local Cohen-Macaulay ring of minimal multiplicity and
$I\subset Q$ is a quasi-Gorenstein ideal?
\end{question}

\begin{disc} 
The ring $R$ from Example~\ref{blah} is strongly homologically of minimal multiplicity,
but is not of minimal multiplicity.
Furthermore, there does not exist a local ring $(Q,\fr)$ with a $Q$-regular
sequence $\x\in\fr^2$ such that $R\cong Q/(\x)Q$;
see~\cite[(3.10)]{gasharov:bpclr}. It follows that there does not
exist a local Cohen-Macaulay ring of minimal multiplicity $(Q,\fr)$ with a $Q$-regular
sequence $\x\in\fr$ such that $R\cong Q/(\x)Q$. However, at this time, we do not
know if there exist a local Cohen-Macaulay ring of minimal multiplicity $(Q,\fr)$ with a 
quasi-Gorenstein ideal $I\subset Q$
such that $R\cong Q/I$. 
\end{disc}

The next two results contain Theorem~\ref{thm02} from the introduction.

\begin{thm}\label{lem43}
Assume that $\psi\colon (R,\m,k)\to (R',\m',k')$ is a 
flat local ring
homomorphism with Gorenstein closed fibre $R'/\m R'$.
If  $R'$
is homologically of minimal multiplicity of type $(m,n,t)$,
then $R$ is homologically of minimal multiplicity of type $(m,n,t)$.
The converse holds when $k$ is perfect and  $R'/\m R'$
is regular.
\end{thm}

\begin{proof}
Assume first that $R'$ is homologically of minimal multiplicity of type $(m,n,t)$
and let $\vf_1\colon R'\to S_1$ be a ring homomorphism and $M_1$ an $S_1$-module
as in Definition~\ref{defn01}. 
The composition
$\vf_1\psi\colon R\to S_1$ 
is flat, and 
Remark~\ref{disc41'} implies  that 
it is Gorenstein. 
It follows readily that this map, with the $S_1$-module $M_1$, satisfies the axioms to show that 
$R$ is homologically of minimal multiplicity of type $(m,n,t)$.

Assume next that $R$ is homologically of minimal multiplicity of type $(m,n,t)$.
Assume further that $k$ is perfect and  $R'/\m R'$
is regular. 
We prove that $R'$ is homologically of minimal multiplicity of type $(m,n,t)$
in two cases.

Case 1: The closed fibre $R'/\m R'$ is a field.
Let $\vf_2\colon (R,\m,k)\to (S_2,\n_2,l_2)$ 
be a ring homomorphism and $M_2$ an $S_2$-module
as in Theorem~\ref{disc11d}.
Since $k'$ and $l_2$ are extension fields of $k$, their join $k''$ 
fits in a commutative diagram of field extensions
$$\xymatrix{k\ar[r]^-{\ol\psi}\ar[d]_{\ol{\vf_2}} & k' \ar[d]^{\alpha_0^{}} \\ l_2 \ar[r]^-{\beta_0^{}} & k''}$$
where $\ol\psi$ and $\ol{\vf_2}$ are induced by $\psi$ and $\vf_2$.
Remark~\ref{disc15} provides flat local ring homomorphisms
$\alpha\colon(R',\m',k')\to(R'',\m'R'',k'')$
and $\beta\colon (S_2,\n_2,l_2)\to (S_3,\n_2S_3,k'')$ 
such that $R''$ and $S_3$ are complete, the map
$k'\to k''$ induced by $\alpha$ is precisely $\alpha_0$, and the map
$l_2\to k''$ induced by $\beta$ is precisely $\beta_0$.

Let $\tau\colon R''\to k''$ and $\pi\colon S_3\to k''$ denote the 
natural surjections. It follows that the small quadrilaterals in the following diagram commute:
$$\xymatrix{
R \ar[rrr]^-{\psi} \ar[ddd]_-{\vf_2}\ar[rd]
&&& R' \ar[ld] \ar[dd]^-{\alpha} \\
& k \ar[r]^-{\ol\psi} \ar[d]_{\ol{\vf_2}} 
& k' \ar[d]^{\alpha_0} \\
& l_2\ar[r]^-{\beta_0} & k'' & R'' \ar[l]_{\tau} \\
S_2 \ar[ru]\ar[rr]^-{\beta} && S_3\ar[u]_-{\pi}
}$$
(The unspecified maps are the canonical surjections.)
It follows that $\tau\alpha\psi=\pi\beta\vf_2$.

Note that the composition $\beta\vf_2\colon R\to S_3$ is flat
because $\beta $ and $\vf_2$ are both flat.
Furthermore, the  closed fibre $S_3/\m S_3$ is regular.
(Indeed, the map $\ol \beta\colon S_2/\m S_2\to S_3/\m S_3$ induced by $\beta$
is flat because $\beta$ is flat. 
The closed fiber of $\ol \beta$ is $S_3/\n_2S_3=k''$, which is regular.
By assumption, the ring $S_2/\m S_2$ is also regular, and thus $S_3/\m S_3$ is regular.)
It follows that the diagram $R\xra{\beta\vf_2}S_3\xra\pi k''$ is a Cohen factorization
of the map $\pi\beta\vf_2$.
Also, since $\n_2S_3$ is the maximal ideal of $S_3$,
its square annihilates the module $M_3=S_3\otimes_{S_2}M_2$, because
$\n_2^2M_2=0$. The canonical module of $S_3$ is 
$\omega_{S_3}\cong S_3\otimes_{S_2}\omega_{S_2}$,
since $\beta$ is flat with Gorenstein closed fibre, and it follows that
$$\tor^{S_3}_i(\omega_{S_3},M_3)\cong S_3\otimes_{S_2}\tor^{S_2}_i(\omega_{S_2},M_2)=0$$
for all $i\geq t$. In particular, the map  $\beta\vf_2\colon R\to S_3$ and $S_3$-module $M_3$ satisfy the
conditions of Definition~\ref{defn01}.

Similarly, the composition $\alpha\psi\colon R\to R''$ is flat
with regular closed fibre, and
the diagram $R\xra{\alpha\psi}R''\xra\tau k''$ is a Cohen factorization
of the map $\tau\alpha\psi$. 
The diagram $R\xra{\beta\vf_2}S_3\xra\pi k''$ is also a Cohen factorization
$\tau\alpha\psi$. 
Since the field $k$ is perfect, the extension $k\to k''$ is separable,
and it follows from~\cite[(1.7)]{avramov:solh} that there is a local ring homomorphism
$\phi\colon R''\to S_3$ making the following diagram commute:
$$\xymatrix{
R\ar[r]^-{\alpha\psi} \ar[d]_-{\beta\vf_2} & R'' \ar[d]^-{\tau}\ar[ld]_-{\phi} \\
S_3 \ar[r]^-{\pi} & k''
}$$
We claim that $\phi$ is flat. To show this, we show that
$\tor^{R''}_i(S_3,k'')=0$ for all $i\geq 1$.
Let $F$ be a free resolution of $k$ over $R$. Since $R''$ is flat over $R$,
the complex $R''\otimes_RP$ is a free resolution of 
$R''\otimes_Rk\cong k''$ over $R''$. It follows that
$$\tor^{R''}_i(S_3,k'')
\cong\HH_i(S_3\otimes_{R''}(R''\otimes_RP))
\cong\HH_i(S_3\otimes_RP)
\cong\tor^R_i(S_3,k)=0$$
for $i\geq 1$; the vanishing comes from the fact that $S_3$ is flat over $R$.

Our assumption that $R'/\m R'$ is a field implies that
the maximal ideal of $R''$ is $\m'R''=\m R''$. Thus, the closed fibre of $\phi$ is
$S_3/\m' S_3=S_3/\m S_3$, which is regular. Hence, the map $\phi\colon R''\to S_3$
with the $S_3$-module $M_3$ satisfies the conditions of Definition~\ref{defn01},
showing that $R''$ is homologically of minimal multiplicity of type $(m,n,t)$.
The local homomorphism $\alpha\colon R'\to R''$ is flat, so the descent result
(established in the first paragraph of this proof) shows that $R'$
is homologically of minimal multiplicity of type $(m,n,t)$.
This completes the proof in this case.

Case 2: the general case.
Let $\x=x_1,\ldots,x_n\in\m'$ be a sequence of elements whose residues modulo
$\m R'$ form a regular system of parameters for the regular ring $R'/\m R'$. 
According to~\cite[Cor.\ of (22.5)]{matsumura:crt}, the sequence
$\x$ is $R'$-regular, and the quotient $R'/\x R'$ is flat as an $R$-module.
Furthermore, the closed fibre of the induced map $R\to R'/\x R'$ is
$R'/(\x R'+\m R')\cong k'$. Since
$R$ is homologically of minimal multiplicity of type $(m,n,t)$,
Case 1 of our proof shows that $R'/\x R'$
is homologically of minimal multiplicity of type $(m,n,t)$.
Since the sequence $\x$ is $R'$-regular,
the descent result in Lemma~\ref{lem42} implies that $R'$
is homologically of minimal multiplicity of type $(m,n,t)$.
\end{proof}

\begin{thm}\label{lem43x}
Assume that $\psi\colon (R,\m,k)\to (R',\m',k')$ is a 
flat local ring
homomorphism with  $R'/\m R'$ Gorenstein.
If  $R'$
is strongly homologically of minimal multiplicity of type $(m,n)$,
then $R$ is strongly homologically of minimal multiplicity of type $(m,n)$.
The converse holds when $k$ is perfect and  $R'/\m R'$
is regular.
\end{thm}

\begin{proof}
The proof is similar to that for Theorem~\ref{lem43}.
\end{proof}

The next question asks if the converses in Theorems~\ref{lem43}
and~\ref{lem43x} hold in general.

\begin{question}
Assume that $\psi\colon (R,\m,k)\to (R',\m',k')$ is a 
flat local ring
homomorphism with Gorenstein closed fibre $R'/\m R'$.
If  $R$
is (strongly) homologically of minimal multiplicity,
must $R'$ be (strongly) homologically of minimal multiplicity?
\end{question}

The next results contain criteria guaranteeing that a 
localized tensor product is (strongly) homologically of minimal multiplicity.

\begin{cor}\label{prop41'}
Let  $(R,\m,k)$ and $(R_1,\m_1,k_1)$ be local $k$-algebras
such that $R_1$ is Gorenstein and $R\otimes_k R_1$ is noetherian. 
Set $P=R\otimes_k\m_1+\m\otimes_k R_1$ and
set $R'=(R\otimes_k R_1)_P$ with maximal ideal $\m'=PR'$.
If $R'$ is homologically of minimal multiplicity of type $(m,n,t)$,
then $R$ is homologically of minimal multiplicity of type $(m,n,t)$. 
The converse holds when $k$ is perfect and $R_1$ is regular.
\end{cor}

\begin{proof}
The natural map $R\to R'$ is flat and local with closed fibre
$$R'/\m R'\cong R/\m\otimes_kR_1\cong k\otimes_kR_1\cong R_1.
$$
The desired conclusions now follow from Theorem~\ref{lem43}.
\end{proof}

\begin{cor}\label{prop41'x}
Let  $(R,\m,k)$ and $(R_1,\m_1,k_1)$ be local $k$-algebras
such that $R_1$ is Gorenstein and $R\otimes_k R_1$ is noetherian. 
Set $P=R\otimes_k\m_1+\m\otimes_k R_1$ and 
set $R'=(R\otimes_k R_1)_P$ with maximal ideal $\m'=PR'$.
If $R'$ is strongly homologically of minimal multiplicity of type $(m,n)$,
then $R$ is strongly homologically of minimal multiplicity of type $(m,n)$. 
The converse holds when $k$ is perfect and $R_1$ is regular.
\end{cor}

\begin{proof}
This is proved similarly to Corollary~\ref{prop41'}.
\end{proof}

The next questions ask if the converses in Corollaries~\ref{prop41'}
and~\ref{prop41'x} hold when $k$ is not perfect or $R_1$ is not regular.

\begin{question}
Let  $(R,\m,k)$ and $(R_1,\m_1,k_1)$ be local $k$-algebras
such that the tensor product $R\otimes_k R_1$ is noetherian. 
Assume that $R_1$ is Gorenstein.
Set $P=R\otimes_k\m_1+\m\otimes_k R_1$, and 
set $R'=(R\otimes_k R_1)_P$ with maximal ideal $\m'=PR'$.
If $R$ is (strongly) homologically of minimal multiplicity,
must $R'$ be (strongly) homologically of minimal multiplicity?
\end{question}

Before continuing, we recall the following handy bookkeeping tool.

\begin{defn}
Given a finitely generated $R$-module $M$,
the \emph{Poincar\'e series} of $M$ is the formal power series
$P_M^R(t)=\sum_{i=0}^\infty\beta_i^R(M)t^i$.
\end{defn}

The following example shows that the local tensor product of two rings that
are strongly homologically of minimal multiplicity need not be homologically of
minimal multiplicity. It also shows that, given a flat local homomorphism
$R\to R'$, if $R$ and $R'/\m R'$ are 
strongly homologically of minimal multiplicity,
then $R'$ need not be homologically of
minimal multiplicity. 

\begin{ex} \label{ex9454}
Assume that $k$ is perfect. 
Set $R=k[X,Y]/(X,Y)^2$ and $R_1=k[Z,W]/(Z,W)^2$.
These are local artinian rings of minimal multiplicity, type 2 and length 3;
see Example~\ref{ex11}.
Hence they are strongly homologically of minimal multiplicity
of type $(1,2)$ by Proposition~\ref{disc11a}. 
Let $\omega$ and $\omega_1$ be canonical modules for $R$ and $R_1$;
their Poincar\'e series are given by the following formula:
\begin{equation} \label{ex9454a}
\textstyle
P^{R}_{\omega}(t)=2+3t\sum_{i=0}^\infty 2^it^i=P^{R_1}_{\omega_1}(t).
\end{equation}
The tensor product $R'=R\otimes_kR_1$ is local because
it is isomorphic to the local ring
$k[X,Y,Z,W]/(X,Y)^2+(Z,W)^2$. From~\cite[(2.5.1)]{jorgensen:jec}, we know that
the canonical module of $R'$ is $\omega'=\omega\otimes_k\omega_1$. Thus,
the K\"unneth formula explains the first equality in the next sequence
\begin{align} 
P^{R'}_{\omega'}(t) \notag
&=P^{R}_{\omega}(t)P^{R_1}_{\omega_1}(t) 
\textstyle
=\left(2+3t\sum_{i=0}^\infty 2^it^i\right)^2 \\
\label{ex9454b}
P^{R'}_{\omega'}(t)
&=\textstyle
4+12t+\sum_{i=2}^\infty (9i+15)2^{i-2}t^i.
\end{align}
The second equality is from equation~\eqref{ex9454a},
and the third one is straightforward.

We show that $R'$ is not homologically of minimal multiplicity.
It suffices to show that there are no integers $r$ and $t$ such that
$\beta^{R'}_{t+s}(\omega')=r^s\beta^{R'}_{t}(\omega')$ for all $s\geq 0$;
see Theorem~\ref{thm01}.
By way of contradiction, suppose that such integers $r$ and $t$ do exist.
Assume without loss of generality that $t\geq 2$.
The first two equalities in the next sequence follow directly,
and the third one is from equation~\eqref{ex9454b}:
\begin{align*}
\frac{\beta^{R'}_{t+1}(\omega')}{\beta^{R'}_{t}(\omega')}
&=r=\frac{\beta^{R'}_{t+2}(\omega')}{\beta^{R'}_{t+1}(\omega')} \\
\frac{[9(t+1)+15]2^{t-1}}{(9t+15)2^{t-2}}
&=\frac{[9(t+2)+15]2^{t}}{[9(t+1)+15]2^{t-1}}\\
[9(t+1)+15]^2&=(9t+15)[9(t+2)+15] \\
(9t+24)^2&=(9t+24)^2-81. 
\end{align*}
The remaining equalities are straightforward consequences;
the final one implies that $0=-81$, a contradiction.

Next, consider the natural map
$R\to R'$, which is flat and local with closed fibre
$R'/\m R\cong R_1$. In particular, the source $R$ and closed fibre
$R'/\m R'$ are strongly homologically of minimal multiplicity,
but the target $R'$ is not.
\end{ex}

The next results describe our most general ascent and descent properties.

\begin{cor}\label{prop41}
Assume that $\psi\colon R\to R'$ is a local, quasi-Gorenstein ring
homomorphism. 
If  $R'$
is homologically of minimal multiplicity of type $(m,n,t)$,
then $R$ is homologically of minimal multiplicity of type $(m,n,t)$.
The converse holds when the residue field $k$ is perfect and when the induced map
$\grave\psi\colon R\to \comp{R'}$ admits a Cohen factorization
$R\xra{\dot\psi}R''\xra{\psi'}\comp{R'}$ such that $\ker(\psi')$
is contained in the square of the maximal ideal of $R''$.
\end{cor}

\begin{proof}
Assume that $R'$
is homologically of minimal multiplicity of type $(m,n,t)$.
Lemma~\ref{lem41}\eqref{lem41a} implies that
$\comp{R'}$
is homologically of minimal multiplicity of type $(m,n,t)$.
Let $R\to R''\to\comp{R'}$ be a Cohen factorization of the induced map
$\grave\psi\colon R\to \comp{R'}$.
Lemma~\ref{lem42} implies that 
$R''$ is homologically of minimal multiplicity of type $(m,n,t)$,
and Theorem~\ref{lem43} yields the same conclusion for $R$.

The converse statement is proved similarly.
\end{proof}

\begin{cor}\label{prop41x}
Assume that $\psi\colon R\to R'$ is a local, quasi-Gorenstein ring
homomorphism. 
If  $R'$
is strongly homologically of minimal multiplicity of type $(m,n)$,
then $R$ is strongly homologically of minimal multiplicity of type $(m,n)$.
The converse holds when  the residue field $k$ is perfect and when the induced map
$\grave\psi\colon R\to \comp{R'}$ admits a Cohen factorization
$R\xra{\dot\psi}R''\xra{\psi'}\comp{R'}$ such that $\ker(\psi')$
is contained in the square of the maximal ideal of $R''$.
\end{cor}

\begin{proof}
This is proved as in Corollary~\ref{prop41}.
\end{proof}

We conclude with some natural questions.

\begin{question}
Assume that $\psi\colon R\to R'$ is a local, quasi-Gorenstein ring
homomorphism. 
If  $R$
is (strongly) homologically of minimal multiplicity,
must $R'$ be (strongly) homologically of minimal multiplicity?
\end{question}

\begin{question}
If  $R$
is (strongly) homologically of minimal multiplicity
and $\p$ is a prime ideal of $R$,
must the localization  $R_{\p}$ be (strongly) homologically of minimal multiplicity?
\end{question}

\section*{Acknowledgments}

We are grateful to 
Srikanth Iyengar, Graham Leuschke, Amelia Taylor, and Yuji Yoshino
for helpful discussions
about this work. We are grateful to the referee for thoughtful comments
about the manuscript.


\begin{thebibliography}{10}

\bibitem{abhyankar:lrhed}
S.~S. Abhyankar, \emph{Local rings of high embedding dimension}, Amer. J. Math.
  \textbf{89} (1967), 1073--1077. \MR{0220723 (36 \#3775)}

\bibitem{avramov:lgh}
L.~L. Avramov and H.-B.\ Foxby, \emph{Locally {G}orenstein homomorphisms},
  Amer. J. Math. \textbf{114} (1992), no.~5, 1007--1047. \MR{1183530
  (93i:13019)}

\bibitem{avramov:rhafgd}
\bysame, \emph{Ring homomorphisms and finite {G}orenstein dimension}, Proc.
  London Math. Soc. (3) \textbf{75} (1997), no.~2, 241--270. \MR{98d:13014}

\bibitem{avramov:solh}
L.~L. Avramov, H.-B.\ Foxby, and B.\ Herzog, \emph{Structure of local
  homomorphisms}, J. Algebra \textbf{164} (1994), 124--145. \MR{95f:13029}

\bibitem{avramov:pmivpd}
L.~L. Avramov, V.~N. Gasharov, and I.~V. Peeva, \emph{A periodic module of
  infinite virtual projective dimension}, J. Pure Appl. Algebra \textbf{62}
  (1989), no.~1, 1--5. \MR{1026870 (90m:13016)}

\bibitem{christensen:gd}
L.~W. Christensen, \emph{Gorenstein dimensions}, Lecture Notes in Mathematics,
  vol. 1747, Springer-Verlag, Berlin, 2000. \MR{2002e:13032}

\bibitem{christensen:scatac}
\bysame, \emph{Semi-dualizing complexes and their {A}uslander categories},
  Trans. Amer. Math. Soc. \textbf{353} (2001), no.~5, 1839--1883.
  \MR{2002a:13017}

\bibitem{christensen:gmirlr}
L.~W. Christensen, J.\ Striuli, and O.\ Veliche, \emph{Growth in the minimal
  injective resolution of a local ring}, J. Lond. Math. Soc., to appear,
  \texttt{arXiv:math/0812.4672v1}.

\bibitem{foxby:gmarm}
H.-B.\ Foxby, \emph{Gorenstein modules and related modules}, Math. Scand.
  \textbf{31} (1972), 267--284 (1973). \MR{48 \#6094}

\bibitem{gasharov:bpclr}
V.~N. Gasharov and I.~V. Peeva, \emph{Boundedness versus periodicity over
  commutative local rings}, Trans. Amer. Math. Soc. \textbf{320} (1990), no.~2,
  569--580. \MR{967311 (90k:13011)}

\bibitem{M2}
D.~R. Grayson and M.~E. Stillman, \emph{Macaulay 2, a software system for
  research in algebraic geometry}, Available at
  \href{http://www.math.uiuc.edu/Macaulay2/}%
  {http://www.math.uiuc.edu/Macaulay2/}.

\bibitem{grothendieck:ega3-1}
A.\ Grothendieck, \emph{\'{E}l\'ements de g\'eom\'etrie alg\'ebrique. {III}.
  \'{E}tude cohomologique des faisceaux coh\'erents. {I}}, Inst. Hautes
  \'Etudes Sci. Publ. Math. (1961), no.~11, 167. \MR{0217085 (36 \#177c)}

\bibitem{hartshorne:lc}
R.\ Hartshorne, \emph{Local cohomology}, A seminar given by A. Grothendieck,
  Harvard University, Fall, vol. 1961, Springer-Verlag, Berlin, 1967.
  \MR{0224620 (37 \#219)}

\bibitem{jorgensen:jec}
D.~A. Jorgensen, \emph{On joins and extension conjectures}, to appear, J.\
  Comm.\ Algebra.

\bibitem{jorgensen:gbscm}
D.~A. Jorgensen and G.~J. Leuschke, \emph{On the growth of the {B}etti sequence
  of the canonical module}, Math. Z. \textbf{256} (2007), no.~3, 647--659.
  \MR{2299575 (2008a:13018)}

\bibitem{matsumura:crt}
H.\ Matsumura, \emph{Commutative ring theory}, second ed., Studies in Advanced
  Mathematics, vol.~8, University Press, Cambridge, 1989. \MR{90i:13001}

\bibitem{reiten:ctsgm}
I.~Reiten, \emph{The converse to a theorem of {S}harp on {G}orenstein modules},
  Proc. Amer. Math. Soc. \textbf{32} (1972), 417--420. \MR{0296067 (45 \#5128)}

\bibitem{sharp:gmccmlr}
R.~Y. Sharp, \emph{On {G}orenstein modules over a complete {C}ohen-{M}acaulay
  local ring}, Quart. J. Math. Oxford Ser. (2) \textbf{22} (1971), 425--434.
  \MR{0289504 (44 \#6693)}

\end{thebibliography}
\providecommand{\bysame}{\leavevmode\hbox to3em{\hrulefill}\thinspace}
\providecommand{\MR}{\relax\ifhmode\unskip\space\fi MR }
\providecommand{\MRhref}[2]{%
  \href{http://www.ams.org/mathscinet-getitem?mr=#1}{#2}
}
\providecommand{\href}[2]{#2}

\end{document}